\documentclass{amsart}
\usepackage[colorlinks, linkcolor=blue,  citecolor=blue, urlcolor=blue]{hyperref}%
\usepackage{url}
\usepackage{verbatim}
\usepackage{subfigure}
\usepackage{graphicx}
\graphicspath{{figures/}}

\newcommand{\norm}[1]{\Vert#1\Vert}

\newcommand{\abs}[1]{\vert#1\vert}

\newtheorem{theorem}{Theorem}
\theoremstyle{definition}
\newtheorem{lemma}[theorem]{Lemma}
\newtheorem{definition}{Definition}

\theoremstyle{remark}

\newtheorem{remark}{Remark}

\DeclareMathOperator{\dv}{div}

\newcommand{\grad}{\nabla}

\newcommand{\bq}{\begin{equation}}
\newcommand{\eq}{\end{equation}}
\newcommand{\R}{\mathbb{R}}

\newcommand{\e}{\epsilon}
\newcommand{\bO}{\mathcal{O}}

\newcommand{\an}{\theta}
\newcommand{\tra}{\mathsf{T}}

\newcommand{\dt}{\rho}
\newcommand{\dx}{h}
\newcommand{\un}{u^{n}}
\newcommand{\unn}{u^{n+1}}
\newcommand{\wa}{\alpha}
\newcommand{\wb}{\beta}

\newcommand{\IL}{\Delta_\infty}
\newcommand{\pLap}{\Delta_p}
\newcommand{\Lap}{\Delta}
\newcommand{\MC}{\Delta_1}
\newcommand{\grhs}{g}

\begin{document}

\title[Finite difference methods for {I}finity {L}aplace equation]
{Finite difference methods for the {I}nfinity {L}aplace and $p$-{L}aplace equations}
\author{Adam M. Oberman}
\address{Department of Mathematics, Simon Fraser University}
\email{aoberman@sfu.ca}
\date{\today} 
\begin{abstract}
We build convergent discretizations and semi-implicit solvers for the Infinity Laplacian and the game theoretical $p$-Laplacian.  The discretizations simplify and generalize earlier ones. We prove convergence of the solution of the Wide Stencil finite difference schemes to the unique viscosity solution of the underlying equation.   We build a semi-implicit solver, which solves the Laplace equation as each step.  It is fast in the sense that the number of iterations is independent of the problem size.
This is an improvement over previous explicit solvers, which are slow due to the CFL-condition.  
\end{abstract}
\thanks{The author is grateful to Selim Esedoglu for suggesting applying Smereka's work on the semi-implicit iteration to this equation}
\maketitle

\section{Introduction}
The Infinity Laplacian equation is at the interface of the fields of analysis, nonlinear elliptic Partial Differential Equations (PDEs), and probabilistic games.  It was first studied in the late 1960s by the Swedish mathematician Gunnar Aronsson \cite{AronssonIL1,AronssonIL2,AronssonIL3},
motivated by classical analysis problem of building Lipschitz extensions of a given function. 
Aronsson found non-classical solutions, but a rigorous theory of weak solutions was not yet available.  It took a few decades until analytical tools were developed to study the equation rigorously, and computational tools were developed which made numerical solution of the equation possible.

In the last decade, PDE theorists established existence and uniqueness, and regularity results.
The theory of viscosity solutions~\cite{CIL} is the appropriate one for studying weak solutions to the PDE.  But the general uniqueness theory did not apply to this very degenerate equation, so proving uniqueness required a new approach.  The first uniqueness result was due to Jensen~\cite{JensenUniq}, followed by a different proof by Barles and Busca~\cite{BarlesBusca}.  
Later, the connection with finite difference equation was exploited by Armstrong and Smart, and they were to give a short uniqueness proof for the PDE \cite{ArmstrongSmartUniqueness} .
The differentiability of solutions remained an open question for some time.  The first result was obtained by \cite{SavinRegIL} in two dimensions, followed by \cite{EvansSavinRegIL}, and \cite{EvansSmartDiffIL} and \cite{EvansSmartAdjointAndFD} in general dimensions.

The first convergent difference scheme was presented in~\cite{ObermanIL}.  
A numerical scheme using the extension property can be found in LeGruyer~\cite{LeGruyer}.  Two different numerical methods were derived in~\cite{EvansSmartAdjointAndFD}, one adapted the monotone scheme in~\cite{ObermanIL} to the standard Infinity Laplacian (which is homogeneous degree two in $\grad u$), the second, quite different, used the variational structure of a regularized PDE.

Earlier work by LeGruyer~\cite{LeGruyerOlder} proved uniqueness for a related finite difference equation.  The proof is a generalization of the uniqueness proof for linear elliptic finite difference schemes from~\cite{MotzkinWasow}.  

A group of probabilists, Peres-Shramm-Sheffield-Wilson~\cite{PeresEtcTugofWar} studying a randomized version of a marble game called Hex found a connection with the Infinity Laplacian equation.   This connection gives an interpretation of the equation as a two player random game.   
This is related to work by Kohn-Serfaty~\cite{KohnSerfatyMC} who found an interpretation of the equation for motion of level sets by mean curvature, \cite{OSnum} \cite{EvansSpruck} as a deterministic two player game.
This equation was also interpreted as a nonlinear average~\cite{ObermanMC} for the purpose of finite difference schemes.    Deterministic game interpretations for more general PDEs followed in ~\cite{KohnSerfatyGeneralGames}.  The connection between these various game interpretations was further studied in~\cite{EvansGames}.

The rich connection between games, finite difference schemes, and nonlinear elliptic PDEs is now much better understood.  There have been a number of works in this area, in particular on the game theoretical p-Laplacian. The probabilistic games interpretation can be found in~\cite{PeresSheffield} (see also~\cite{ManfrediPLap}).  Related works include biased games which corresponds to a gradient term~\cite{ArmstrongSSMixedIL}

This article will further exploit the connection between games, finite difference schemes, and nonlinear elliptic PDE, by building convergent finite difference schemes which are consistent with the game interpretation. 
The existence and uniqueness results are now established, efficient numerical solution of the equation remains a challenge.  The original convergent scheme proposed in~\cite{ObermanIL} converged, but is not efficient: as the grid size grows, so does the number of iterations required to find the solution.   This article improves and simplifies the original discretization, and also finds fast solution methods.  It also generalizes the scheme and the solvers to the game-theoretical $p$-Laplacian, which is a convex combination of the Laplacian and the Infinity Laplacian.

\subsection{Introduction to numerical methods for degenerate elliptic PDEs}
There are two major challenges in building numerical solvers for nonlinear and degenerate elliptic Partial Differential Equations (PDEs).  The first challenge is to build convergent approximations, usually by finite difference schemes.  The second challenge is to build efficient solvers.

The approximation theory developed by Barles and Souganidis~\cite{BSNum}  provides criteria for the convergence of approximation schemes:  monotone, consistent, and stable schemes converge to the unique viscosity solution of a degenerate elliptic equation.  But this work does not indicate how to build such schemes, or how to produce fast solvers for the schemes.  It is not obvious how to ensure that schemes satisfy the comparison principle.  The class of schemes which for which this property holds was identified in~\cite{ObermanSINUM}, and were called \emph{elliptic}, by analogy with the structure condition for the PDE.     

An important distinction for this class of equations is between first order (Hamilton-Jacobi) equations, and the second order (nonlinear elliptic) case.  The theory of viscosity solutions~\cite{CIL} covers both cases, but the numerical methods are quite different. In the first order case, where the method of characteristics is available, there are some exact solutions formulas (e.g. Hopf-Lax) and there is a connection with one dimensional conservation laws~\cite{EvansBook}.  The second order case has more in common with divergence-structure elliptic equations, but because of the degeneracy or nonlinearity, many of the tools from the divergence-structure case (e.g. finite elements, multi grid solvers) have not been successfully applied.

In the first order case, there is much more work on discretizations and fast solvers.
For Hamilton-Jacobi equations, which are first order nonlinear PDEs, monotonicity is necessary for convergence. Early numerical papers studied explicit schemes for time-dependent equations on uniform grids~\cite{CrandallLionsNum, SougNum}.   These schemes have been extended to higher accuracy schemes,  which include second order convergent methods, the central schemes~\cite{LinTadmorHJ}, as well as higher order interpolation methods, the ENO schemes~\cite{OsherShuENO}.  Semi-Langrangian schemes take advantage of the method of characteristics to prove convergence~\cite{FalconeSemiLagrangian}.  These have been extended to the case of differential games~\cite{FalconeBardiGames}. 
Two classes of fast solvers have been developed, fast marching~\cite{FastMarching}, and  fast sweeping~\cite{FastSweeping}, The fast marching and fast sweeping methods give fast solution method for first order equations: both take advantage of the method of characteristics, which is not available in the second order case.  

There is much less work in the second order degenerate elliptic equations.  
The equation for motion by mean curvature~\cite{OSnum},~\cite{EvansSpruck} has been extensively studied.
There is an enormous literature on this equation, but we just closely related references.  The connection with games was already mentioned above.  Numerical schemes include~\cite{CarliniFalconeFerrettMC} and~\cite{FalconeFerrettiCarliniMC} .  
In the case of motion by mean curvature, the equation is time-dependent, so a fast solver would allow larger time steps.  For this equation, a semi-implicit solver has been built by Smereka~\cite{Smereka}.  The idea from the Smereka paper will be adapted in this work to build fast solvers for Infinity Laplace.
Another equation in this class is the Hamilton-Jacobi-Bellman equations, for the value function of a stochastic control problem.  Applications include portfolio optimization and option pricing in mathematical finance.
Numerical works include the early paper~\cite{LionsMercierHJB} and~\cite{FalconeCamilliHJB}, 
and a paper on fast solvers~\cite{ZidaniFastHJB}.

For uniformly elliptic PDEs, monotone schemes are not \emph{necessary} for convergence (for example most higher order Finite Element Methods are not monotone).  
But for fully nonlinear or degenerate elliptic, the only convergence proof currently available requires monotone schemes.   One way to ensure monotone schemes is to use Wide Stencil Finite difference schemes, this has been done for the equation for motion by mean curvature,~\cite{ObermanMC}, for the Infinity Laplace equation~\cite{ObermanIL}, for functions of the eigenvalues~\cite{ObermanEigenvalues}, for Hamilton-Jacobi-Bellman equations~\cite{ZidaniWideHJBLong},~and for the convex envelope~\cite{ObermanCEnumerics}.  
Even for linear elliptic equations, a wide stencil scheme maybe necessary for to build a monotone scheme~\cite{MotzkinWasow}.
In some cases, simple finite difference schemes, with minor medications, can give good results, as is the case for the Monge-Amp\`ere equation~\cite{BenamouFroeseObermanMA}.  But we show below that simple finite difference schemes are not convergent for the Infinity Laplace equations.   

The second challenge, which is quite distinct from the first, is to build \emph{solvers} for the finite difference schemes.  For fixed values of $\dx$, the finite difference scheme is a finite dimensional nonlinear algebraic equation which must be solved.  Building solvers demands very different techniques, and little progress has been made, in part, due to the fact that the discrete equations can be non-differentiable, which precludes the use of the Newton's method.   To date, the only general solver available is a fixed point iteration, which corresponds to solving the parabolic version of the equation for long time.  This method is restricted by a nonlinear version of the CFL condition~\cite{ObermanSINUM}, which means the number of iterations required to solve the equation increases with the problem size.   For the Monge-Amp\`ere equation, fast solvers have been built using Newton's method~\cite{ObermanFroeseMATheory}~\cite{ObermanFroeseFast}, but this equation has a different structure (convex, differentiable) from the Infinity Laplace (or the $p$-Laplace) equation.

\subsection{Contribution of this work}
The first contribution of this work is to build a provably convergent discretization of the operator.  The issue here is to ensure that the discretization convergences (in the limit of the discretization parameters going to zero) to the unique viscosity solution of the PDE.   Simply using standard finite differences fails to converge, as shown below. 

The appropriate notion of weak solutions for the PDE is provided by viscosity solutions~\cite{CIL, CrandallTour}.  The only schemes which can be proven to converge to viscosity solutions are monotone schemes~\cite{BSNum}; these schemes satisfy the maximum principle at the discrete level~\cite{ObermanSINUM}.   
For the variational $p$-Laplacian, Galerkin Finite Element methods could be used.  But the game-theoretical version is not a divergence structure operator, so there is not a natural version of weak solutions.
Monotone schemes can be proven to converge for the game-theoretical p-Laplacian~\eqref{pLapId}.  We prove convergence of the solution of the Wide Stencil finite difference schemes to the unique viscosity solution of the underlying equation~\eqref{pLapId}~\eqref{D}. 

The second contribution of this work is to build fast solvers for~\eqref{pLapId}.
There are two reasonable ways to quantify the notion of a fast solver.
The first notion of speed is absolute: the number of operations to solve the equation the should be proportional to the problem size.
The second notion of speed is relative: we compare the speed of our solvers to the speed of solvers for a related but easier problem.  Here, it is natural to compare with the solution speed of the Laplace equation.  

Explicit solvers are available and simple to implement, but they are not fast. 
Any monotone scheme can be solved using an iterative, explicit method~\cite{ObermanSINUM}.  The explicit method can be interpreted as a Gauss-Seidel solver, or the forward Euler method for the equation $u_t = \pLap$.  However the time step for the Euler method is $\bO(\dx^2)$, where $\dx$ is the spatial resolution.  The explicit  method is not fast because the number of iterations required for it to converge is $\bO(1/\dx^2)$, which increases with the problem size.

The method we propose is semi-implicit, with the implicit step given by solving the Laplace equation.   The Laplace equation can be solved in $\bO(N)$ operations, using Fast Fourier Transforms, or $\bO(N\log N)$, using sparse linear algebra.   Thus, to the extent needed for our rather coarse analysis, both notions of speed coincide, provided the solution is obtained in a finite (small) number of iterations.  

\subsection{The setting for the PDE}
This work is concerned with the efficient numerical solution of a nonlinear, degenerate elliptic Partial Differential Equation (PDE), the normalized Infinity Laplacian.  The PDE operator is given by
\bq\label{IL} \tag{IL}
\IL u = \frac{1}{\abs{\grad u }^2}\sum_{i,j=1}^d u_{x_i} u_{x_i x_j} u_{x_j}
\eq
where $u(x) :\R^d \to \R$.  

We also study a closely related PDE, the game theoretical p-Laplacian,
 which interpolates between the $1$-Laplacian, $\MC$,  
 \bq\label{MC}\tag{MC}
\MC u = \abs{\grad u} \dv( \grad u / \abs{\grad u})
= \Lap u - \IL u
\eq
and the infinity Laplacian, $\IL$.
Expanding the $\MC$ operator above leads to the identity
\bq\label{LapILMC}
\Lap = \IL + \MC,
\eq
which we record for future use.
The game theoretical $p$-Laplacian is the $p$ weighted average of the $1$- and $\infty$-Laplacians,
\begin{align}\label{pLapId}\tag{pLap}
\pLap =  \frac 1 p \MC + \frac 1 q \IL, \quad p^{-1} + q^{-1} = 1.
\end{align} 
This is consistent with the definitions given in the probabilistic games interpretation~\cite{PeresSheffield} (see also~\cite{ManfrediPLap}).
The normalized versions of the operators are also used in image processing~\cite{CasellesImage}, \cite{CongShapePLap}, \cite{SapiroWarping}.  

Special cases occur for $p=1$ and $\infty$, as above, and for $p=2$ we obtain
$
\Delta_2 = \frac 1 2 \Delta.
$
Using the identity~\eqref{LapILMC} we can also write
\bq\label{plap2}
\pLap = \wa \Lap + \wb \IL, \quad \wa = 1/p,~ \wb = (p-2)/p
\eq
Equation~\eqref{plap2} will be used for $p \in [2,\infty]$.  If we were to consider the case $p \in [1,2]$, the equation above is not a positive combination of the operators.  Instead, for the case 
$p \in [1,2]$, the corresponding representation would be a convex combination of the monotone discretization of $\MC$ and the Laplacian.  Here we will focus on the case where the Infinity Laplace operator is active.

We consider the Dirichlet problem for the operator, in a domain $\Omega \subset \R^d$, with a given right hand side function~$\grhs$.
\begin{align}
\label{PDE}\tag{PDE}
\pLap
u(x) &= \grhs(x),  &\text{ for } x \in \Omega,
\\
\label{D}\tag{D}
u(x) &= h(x), & \text{ for } x \in \partial\Omega.
\end{align}
Here $\grhs$ represents a running cost for a probabilistic game~\cite{PeresSheffield}. When $\grhs = 0$, the operator coincides with the variational $p$-Laplacian.   The relationship between the game theoretical $p$-Laplacian and the variational $p$-Laplacian is given below. 


\subsection{Failure of the standard finite difference scheme}

Here we motivate the need for a convergent scheme, by showing that the standard finite difference scheme fails to converge.  

A natural scheme is given by standard finite differences, along with a small regularization for the norm of the gradient.  For this we use standard centered finite differences for $u_{xx}, u_{yy}, u_{x}, u_y$, and the with the symmetric scheme for $u_{xy}$,
\begin{align*}
u_{x}(x,y)  =&  \frac{1}{2h} \left  (
u(x+h,y) - u(x-h,y)
\right ) + \bO(h^2),
\\
u_{xx}(x,y)  =& \frac{1}{h^2} \left  (
u(x+h,y) -2u(x,y) + u(x-h,y)
\right ) + \bO(h^2),
\\
u_{xy}(x,y) =& 
+\frac{1}{4h^2} \left  (
u(x+h,y+h) + u(x-h,y-h) 
\right ) 
\\
&
- 
\frac{1}{4h^2} \left  (
u(x-h,y+h) + u(x+h,y-h)
\right ) 
+ \bO(h^2)
\end{align*}
and similarly for the $u_y, u_{yy}$ terms.  In order to regularize the gradient, we replaced $\norm{\grad u}^2$ with $\max\{\dx^2, \norm{\grad u}^2\}$. 

We computed the solution with boundary conditions corresponding to the exact Aronsson solution~\cite{AronssonIL2}.   
The finite difference scheme presented above fails to converge, see Figure~\ref{figBadSoln}.  In this case, the solution has the form $\abs{x} - \abs{y}$ in the centre.    In fact, it can be shown using symmetry considerations that $\abs{x} - \abs{y}$ is an exact solution of the symmetric finite difference scheme. On the flat parts, the operator is zero, so we only need to check the corners.  In fact, $u(x,y) = \abs{x}$ and $u(x,y) = \abs{y}$ are also exact solutions.  While other discretizations are possible which break this symmetry, we tried several other simple consistent finite difference schemes and were always able to find examples where they failed to converge.

\begin{figure}[htbp]
\begin{center}
\includegraphics[width=.49\textwidth]{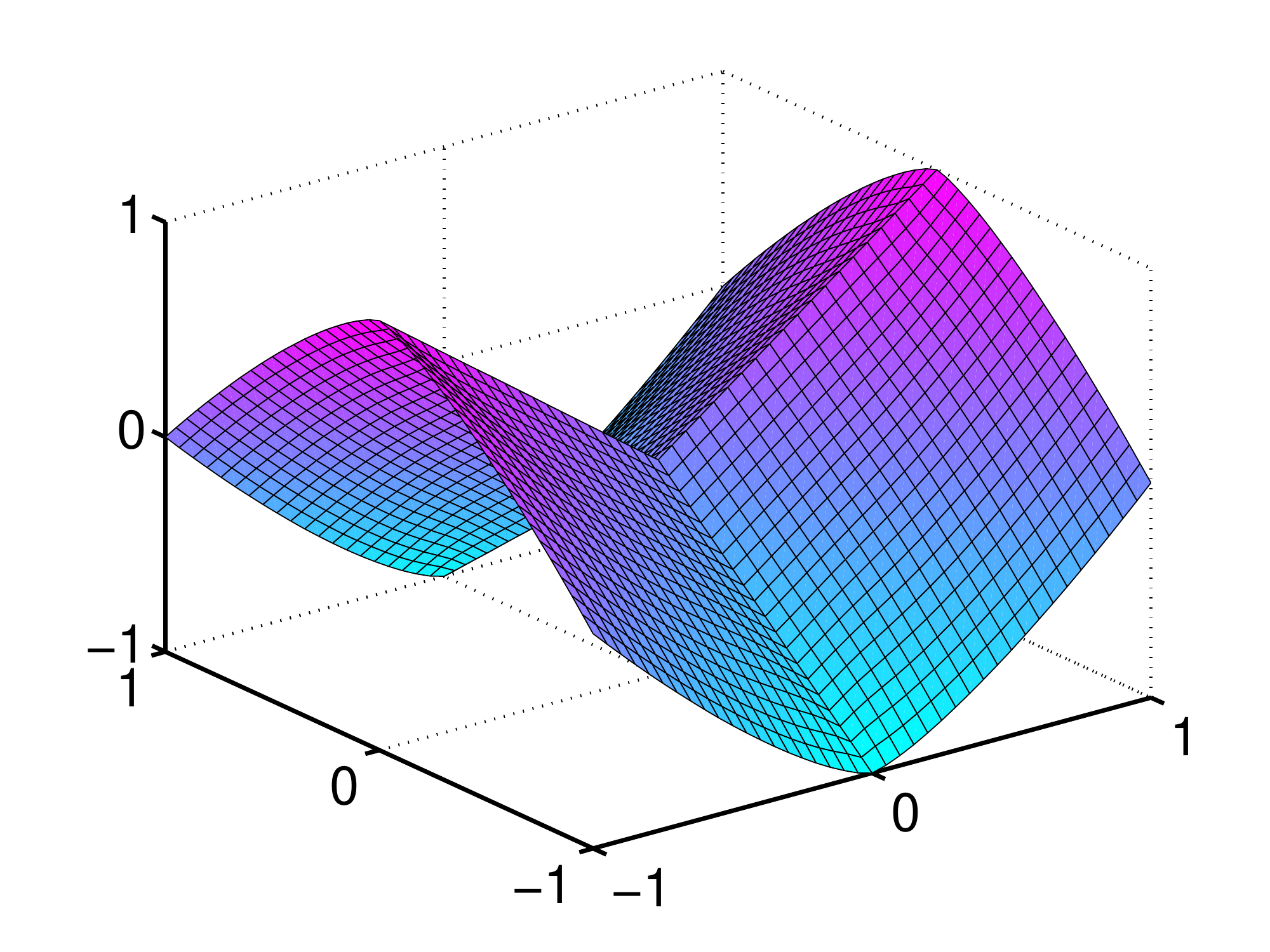}
\includegraphics[width=.49\textwidth]{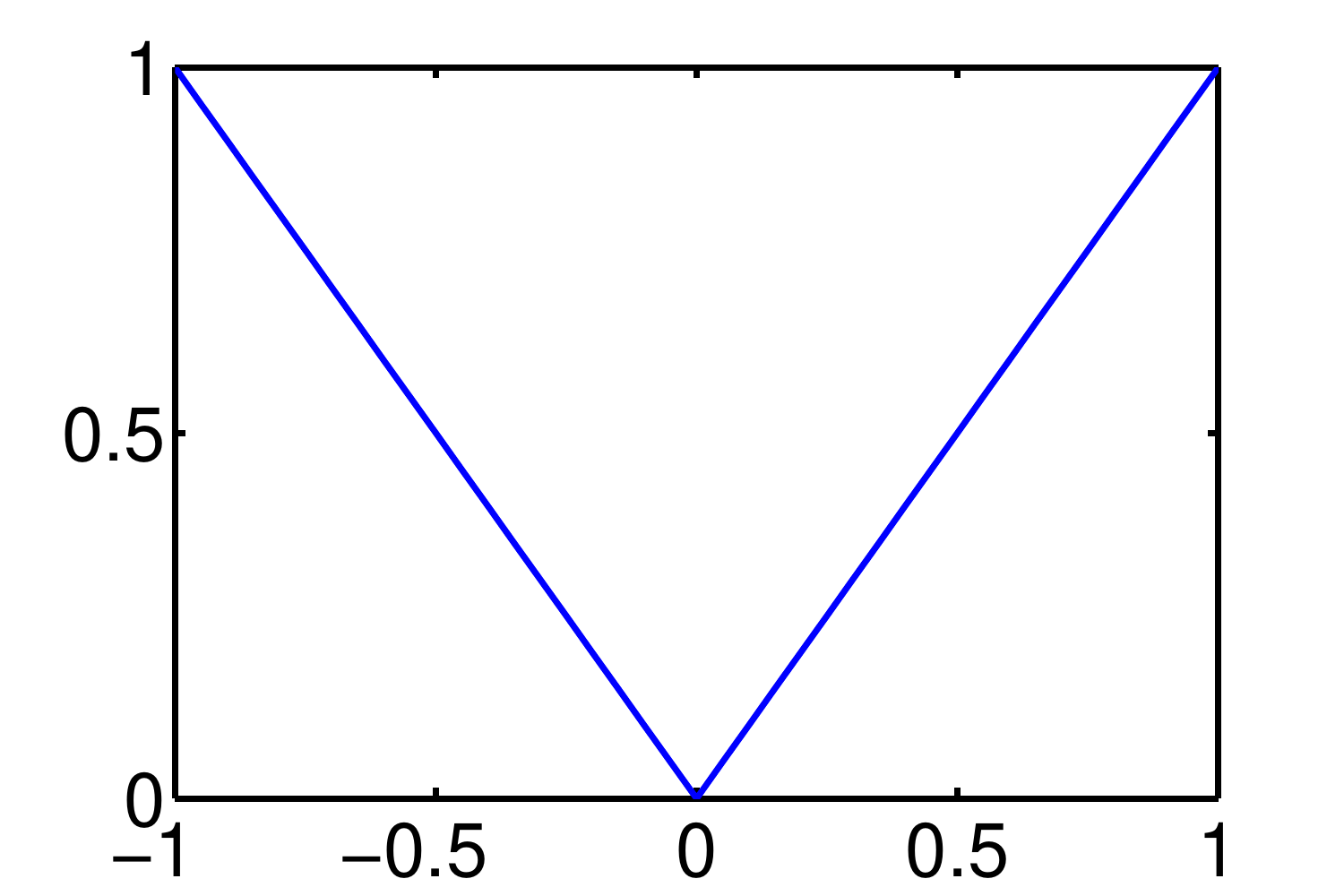}
\caption{
Standard finite difference solutions of the $\infty$-Laplacian.
Boundary data is a  $\abs{x}^{4/3}-\abs{y}^{4/3}$, the computed solution is incorrect, with a singularity of the form  $\abs{x} - \abs{y}$ at the origin.  
(a) Surface Plot, (b) plot of $u(x,0)$.
}
\label{figBadSoln}
\end{center}
\end{figure}

\subsection{Variational $p$-Laplacian}
To clarify any confusion, we also discuss the variational $p$-Laplacian.  This operator arises in the variational problem
\[
J[u] = \int_\Omega  
\abs{\grad u(x)}^p\, dx, \quad u = h \text{ on } \partial \Omega
\]
for $1 < p < \infty$.  The \emph{Euler-Lagrange equation} for the minimizer is
\[
\dv \left (\abs{\grad u(x)}^{p-2} \grad u \right) = 0.
\]
The game theoretical $p$-Laplacian,~\eqref{pLapId}, is related to the variational version by a normalization which makes it homogeneous of order zero in the norm of the gradient.
\bq\label{pLap}
\pLap  u = \frac{1}{p \abs{\grad u }^{p-2}} \dv \left (\abs{\grad u(x)}^{p-2} \grad u \right) 
\eq
To check consistency with definition~\eqref{pLapId}, expand the two terms in the operator above to give
$
\pLap  u = \frac 1 p \Lap u + \frac{(p-2)}{p} \IL u
$
which is the representation~\eqref{plap2}.

When the right hand side function, $\grhs$, is zero, the solutions of both versions of the $p$-Laplacian equations coincide.



\section{Viscosity solutions}\label{sec:viscositysolutions}
In this section we recall the definition of viscosity solutions for the infinity Laplacian and the definition of consistency used in the convergence theory.
\begin{definition}[Viscosity solutions]
(1) A continuous function $u$ defined on the set $U$, is a viscosity \emph{subsolution} of $-\IL u = 0$ in $U$, if for every local maximum point $x\in U$ of $u-\phi$, where $\phi$ is $C^2$ in some neighborhood of $x$, we have
\[
\begin{cases}
-\IL \phi(x) \le 0,               & \text{ if } D\phi(x) \ne 0,\\
-\eta^T D^2 \phi(x) \eta \le 0   &\text{ for some } |\eta| \leq 1, \text{ if } D\phi(x) = 0. \end{cases}
\]
(2) A continuous function $u$ defined on the set $U$, is a viscosity \emph{supersolution} of $-\IL u = 0$ in $U$, if for every local maximum point $x\in U$ of $u-\phi$, where $\phi$ is $C^2$ in some neighborhood of $x$, we have
\[
\begin{cases}
-\IL \phi(x) \ge 0,               & \text{ if } D\phi(x) \ne0,\\
-\eta^T D^2 \phi(x) \eta \ge 0   &\text{ for some } |\eta| \leq 1, \text{ if } D\phi(x) = 0. \end{cases}
\]
(3) Moreover, a continuous function defined on the set $U$, is a viscosity \emph{solution} of $-\IL u = 0$ in $U$, if it is both a viscosity subsolution and a viscosity supersolution in $U$.
\end{definition}

Consistency requires only that we can apply the test function definition in the limit.  

\newcommand{\Sk}{\IL}
\begin{definition}[Consistency]
The numerical scheme $\Sk^{dx,d\theta}$ is \emph{consistent} if for every $\phi \in C^2(U)$, and for every $x\in U$,
\[
\lim_{dx,d\theta\to 0} \Sk^{dx,d\theta}(\phi)(x) = -\IL \phi(x)
\]
if $D\phi(x) \ne0$, and
\bq\label{conspzero}
\lambda \leq \liminf_{dx,d\theta\to 0} \Sk^{dx,d\theta}(\phi)(x) \leq \limsup_{dx,d\theta\to 0} \Sk^{dx,d\theta}(\phi)(x) \leq \Lambda
\eq
where $\lambda, \Lambda$ are the least and greatest eigenvalues of $D^2\phi(x)$, otherwise.
\end{definition}
By a theorem of Barles-Souganidis \cite{BSNum}, consistent, monotone schemes converge to the viscosity solution of the PDE, provided this solution is unique.

\section{Discretization}
In this section we present the discretization of the Infinity Laplace operator, which is needed for convergence to the viscosity solution.  The discretization we present here is different from the one in~\cite{ObermanIL}. The previous scheme was given by solving the a discrete version of the Lipschitz extension problem.  This scheme is simpler, since the resulting equation is explicit.  In addition, this scheme gives the correct scaling in $\dx$ which is needed for a non-zero right hand side.

 By now it is well known that,  for smooth functions with non-vanishing gradient,
the operator is approximated by the average below.  This result follows from Taylor expansions, and the fact that the minimum (or maximum) is in the direction of the gradient, at least to $O(\dx)$.  We show below that the accuracy is actually $\bO(\dx^2)$, which is an improvement over previous results.

\begin{lemma}\label{lem:consistency}
Let $u(x)$ be a smooth function with non-vanishing gradient at $x$.  Then 
\bq\label{ILallBall}
\Delta_\infty u(x) =
\min_{\abs{y-x} = \e} 
\frac{u(y) - u(x) }{\e^2}
+ 
\max_{\abs {y-x} = \e} 
\frac{u(y) - u(x) }{\e^2}
+ O(\e^2).
\eq
\end{lemma}
\begin{proof}
We prove the result in two dimensions, which is all that is needed here.  A longer proof is possible which in higher dimensions, which requires a Lagrange multiplier, $\lambda$ for the constraint, and an asymptotic expansion in $\lambda$ as well.

It is enough to consider $u(x) = p^\tra x + \frac 1 2  x^\tra Qx$, for nonzero $p = \grad u(x)$, and symmetric matrix $Q = D^2u(x)$.  Use the notation
$\hat{p} = \frac{p}{\abs{p}}$, $p^\perp = (-p_2, p_1),$ for $p = (p_1, p_2)$.
With this notation, the operator,~\eqref{IL} is given by
\begin{align}
\label{ILquadratic}
\IL u  &= (\hat{p})^\tra Q \hat{p}
\end{align}

Write 
\[
F(\an) = {u(\e x(\an))} =  \e p^\tra x(\an) + \frac {\e^2} 2  x(\an)^\tra Qx(\an), \qquad x(\theta) = (\cos(\an),\sin(\an))
\]
Then a critical point of $F$ is given by 
\[
0 = F'(\an) = ( p + \e Qx(\an))^\tra x'(\an) 
= 
( p + \e Qx)^\tra x^\perp,
\]
where $x'(\an) = x^\perp(\an) = (-\sin(\an), \cos(\an))$.
Perform an asymptotic expansion the condition in $\e$, with
$x = x_0 + \e x_1$,
 to obtain
\[
(p + \e Qx_0)^\tra (x_0 + \e x_1)^\perp +   \bO(\e^2) = 0.
\]
Then the $\bO(1)$ terms give
\[
p^\tra x_0^\perp = 0
\]
which yields
\[
x_0 = \pm \hat p, 
\]
and the $\bO(\e)$ terms give
\[
p^\tra x_1^\perp + (x_0^\perp Q)^\tra x_0 = 0
\]
which yields
\[
x_1 = - \frac{ (\hat p^\perp)^\tra Q \hat p}{\abs{p}} \hat p^\perp
\]
(Note that we are violating the constraint that $x$ be a unit vector, but the constraint is still satisfied to $\bO(\e^2)$).
Write 
\[
x^+ = \arg \max_\an F(\an),
\qquad
x^- = \arg \min_\an F(\an)
\]
Then we have 
\[
x^\pm = \pm \hat p + \e c {\hat p^\perp}, \quad c = - \frac{ (\hat p^\perp)^\tra Q \hat p}{\abs{p}}. 
\]
Inserting the values for $x^\pm$, into the expression on the right hand side of~\eqref{ILallBall} gives
\begin{align*}
\min_{\abs{y-x} = \e} 
\frac{u(y) - u(x) }{\e^2}
+ 
\max_{\abs {y-x} = \e} 
\frac{u(y) - u(x) }{\e^2}
&= \frac{ F( \e x^-) + F( \e x^+)}{\e^2}
\\
&= \hat p^\tra Q \hat p +  \e^2  c^2  \left((\hat p^\perp)^\tra Q  \hat p^\perp \right)
\\
&=  \IL u +  \e^2  \left ( \frac{ (\hat p^\perp)^\tra Q \hat p}{\abs{p}}  \right )^2 \MC u
\\
&= \IL u + \e^2 c^2 \MC u
\end{align*}
which gives the desired result.
\end{proof}

However, for any given grid, it is impossible to sample the values on the entire circle.  Instead, only values in a discrete set of directions are available.  While it is certainly possible to interpolate the values onto the ball, quadratic interpolation is not monotone, so it violates the maximum principle, which is needed for the convergence proof.  As we show in an example below, non-monotone schemes do not converge for this equation.

So an additional discretization parameter is needed, which we present in what follows. 
\begin{definition}[Spatial and directional resolution]
Given a stencil of neighbouring grid points $v_1,\dots,v_n$ on a Cartesian grid, 
define the local \emph{spatial resolution}, $\dx$,  to be the maximum length of the neighbours
\bq\label{dxdefn}
\dx = \max_{i=1}^n |v_i|,
\eq
and the local \emph{directional resolution}, $d\theta$, to be the maximum directional distance to a neighbour
\bq\label{dthetadefn}
d\theta = \max_{\abs{v} = 1} \min_{i=1}^n \left | v - \frac{v_i}{\abs{v_i}} \right |
\eq
\end{definition}
The direction vectors used will be on a grid, arranged as in~\autoref{fig:schemes}.  In practice, we obtain acceptable accuracy using a relatively narrow stencil. 

\begin{figure}
\centering
 \scalebox{.8}{\includegraphics{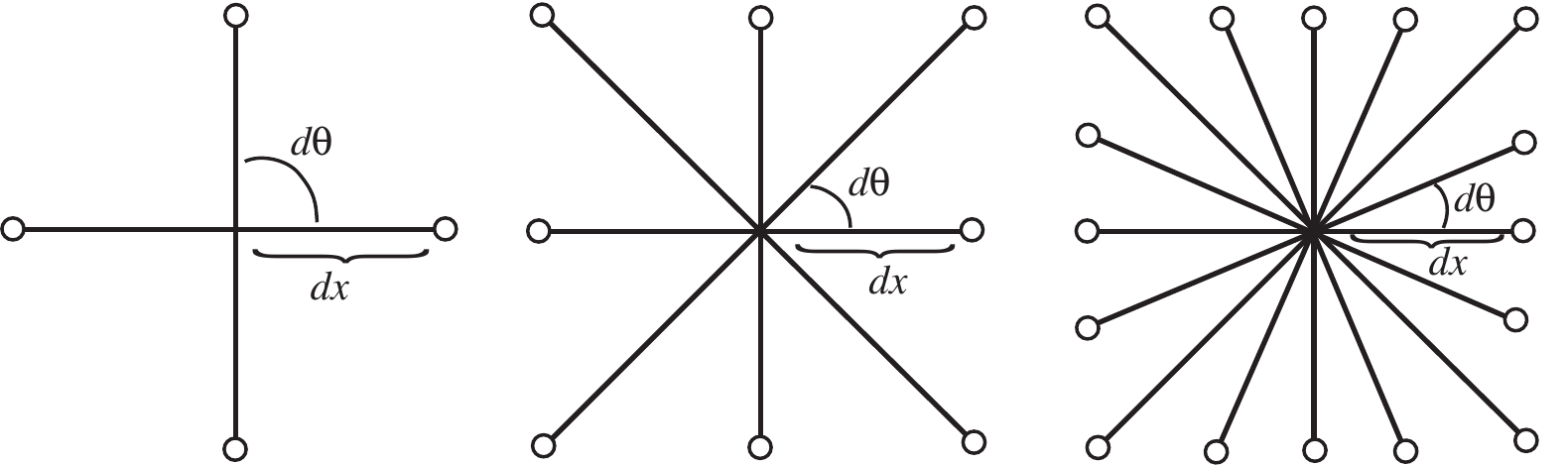} }
\caption{Stencils for the 5, 9, and 17 point schemes} \label{fig:schemes}
\end{figure}

\begin{definition}[Scheme definition]
Define the discretization of $\IL$ to be given by 
\bq\label{ILdxdth}
\IL^{\dx,d\theta}  u(x)  = \max_{i}  \frac{u(x+v_i) - u(x)}{\abs{v_i}^2} + \min_{i}  \frac{u(x+v_i) - u(x)}{\abs{v_i}^2} 
\eq
where $\{v_i\}$ are the neighbours of the point $x$, as in~\autoref{fig:schemes}.
\end{definition}

Next we prove consistency, when the parameters $dx,d\theta$ go to $0$.
For the grid points, we will assume
\begin{gather}
\label{SymmetricDirections}
\text{whenever $v_i$ is a neighbouring grid point, $-v_i$ is as well},\\
\label{dxprime}
\frac{ \max_{i} \abs{v_i}  } {\min_{i} \abs{v_i}} = o(1)
\end{gather}

\begin{theorem}\label{thm:consistency}
Let $u$ be a smooth function in a neighbourhood of $x$, then 
\bq\label{ILconsistency}
\IL^{\dx,d\theta}  u(x) = \IL u(x) + \bO(d\theta + \dx^2)
\eq
\end{theorem}
\begin{proof}
It is enough to consider 
\[
u(x) = p^\tra x + \frac 1 2  x^\tra Qx,
\]
for $p = \grad u(x)$, and symmetric matrix $Q = D^2u(x)$.  Use the notation $\hat{p} = \frac{p}{\abs{p}}$.

First consider the case $Du \not = 0$.
Define 
\[
v^+ = \arg \max_{i}  \frac{u(x+v_i) - u(x)}{\abs{v_i}^2}
\]
\[
v^- = \arg \min_{i}  \frac{u(x+v_i) - u(x)}{\abs{v_i}^2}
\]
Using a Taylor expansion for $u$, we have
\[
\frac{u(x+v_i) - u(x)}{\abs{v_i}^2} = \frac{v_i p + C |v_i|^2}{\abs{v_i}^2}
\]
so, as $\dx \to 0$, the maximum occurs at $v^+ = \arg\max_i \hat{v}_i p$ which is the closest direction vector to $\hat p$, so 
\bq\label{vp}
\widehat {v^+} = \hat p + O(d\theta)
\eq
by~\eqref{dthetadefn} and similarly
\bq\label{vn}
\widehat{ v^-} =\hat p + O(d\theta),
\eq
In addition, since according to~\eqref{SymmetricDirections} we assume that the grid points are arranged symmetrically, we have 
\bq\label{vpvn}
\widehat {v^+} = -\widehat {v^-} 
\eq
as $\dx, d\theta \to 0$.

Then using the Taylor expansion for $u$, insert~(\ref{vp},\ref{vn}) into \eqref{ILdxdth} to obtain
\begin{align*}
\IL^{\dx,d\theta}  u(x) 
&= \left(\frac{v^+}{\abs{v^+}^2} + \frac{v^-}{\abs{v^-}^2} \right) p + \frac 1 2 \frac{v^+ Q v^+}{\abs{v^+}^2} + \frac{v^- Q v^-}{\abs{v^-}^2} + \bO(\dx) &
\\ &= \frac 1 2 \frac{v^+ Q v^+}{\abs{v^+}^2} + \frac 1 2 \frac{v^- Q v^-}{\abs{v^-}^2}+ \bO(\dx) & \text{ by \eqref{vpvn}}
\\ &= \hat p Q \hat p + \bO(\dx+ d\theta) & \text{ by (\ref{vp},\ref{vn})}
\end{align*}
which is consistent, and accurate to $\bO(\dx+ d\theta).$

In the case $Du(x) = 0$, we are only required only to verify \eqref{conspzero}.  In this case we can assume $u(x) = \frac 1 2  x^\tra Qx$, and compute
\[
\max_{i}  \frac{u(x+v_i) - u(x)}{\abs{v_i}^2} = \frac 1 2 \max_i \widehat{ v_i}^\tra Q  \widehat{ v_i} = \frac 1 2 \Lambda + \bO(d\theta + \dx^2)
\]
and similarly
\[
\min_{i}  \frac{u(x+v_i) - u(x)}{\abs{v_i}^2} = \frac 1 2 \min_i \widehat{ v_i}^\tra Q  \widehat{ v_i} = \frac 1 2 \lambda +  \bO(d\theta + \dx^2)
\]
where $\lambda, \Lambda$ are the smallest and largest eigenvalues of $Q$, respectively.  So then
\[
\IL^{\dx,d\theta}  u(x) = \frac 1 2 (\lambda + \Lambda) +  \bO(d\theta + \dx^2)
\]
which is consistent, according to~\eqref{conspzero}.
\end{proof}

\subsection{The discretization of the $p$-Laplacian}
We wish to build a consistent, monotone scheme for 
\[
\pLap = \wa \Lap + \wb \IL.
\]
Starting with the scheme for $\IL$, we can simply combine this with the standard finite differences for the Laplacian,
\begin{gather}\label{LapFD}
\Lap^\dx u(\cdot )  =  4\frac{\bar u - u(\cdot) }{\dx^2}, 
\\ \nonumber
\bar u(x,y) = \frac 1 4  \left( u(x+h,y) + u(x-h,y) + u(x,y+h) + u(x,y-h)\right)
\end{gather}

  Then, using the characterization of monotone schemes from~\cite{ObermanSINUM}, the combined scheme is still monotone.

\begin{theorem}[Convergence]  The solution of the difference scheme for the $p$-Laplacian,~\eqref{plap2},  which is given by a convex combination of the schemes for~\eqref{IL}, ~\eqref{ILdxdth} and the standard finite difference scheme for the Laplacian,~\eqref{LapFD}
\[
\pLap^{\dx,d\theta}
= \wa \Lap^\dx + \wb  \IL^{\dx,d\theta}
\]
converges (uniformly on compact sets) as $\dx,d\theta \to 0$ to the solution of \eqref{pLap}.
\end{theorem}
\begin{proof}  The uniform convergence of solutions of consistent, monotone schemes follows from the main result of~\cite{BSNum}, provided solutions are unique.  The uniqueness follows in our case (although not in the case $p=1$).   Thus by appealing to this result we need only establish consistency and monotonicity of the schemes.

Consistency of the discretization of the $\IL^{\dx,d\theta}$ operator follows from~\autoref{thm:consistency}.
Consistency of the discretization for~\eqref{pLap} follows since it is a convex combination of $\IL$ and the consistent discretization of the Laplacian~\eqref{LapFD}.  

 The definition~\eqref{ILdxdth} expresses the scheme $\IL^{\dx,d\theta}$ as a nondecreasing combination of differences between $u(x_i) - u(x)$, where $x_i$ are neighbouring grids points to the reference point $x$.  This characterizes \emph{elliptic} schemes, which were proven to be monotone in~\cite{ObermanSINUM}.
Monotonicity of the discretization of the $\pLap$ operator is a consequence of the fact that it is a convex combination of two elliptic schemes is also elliptic, which was also shown in~\cite{ObermanSINUM}.  

Together these results prove convergence.
\end{proof}

\section{Solvers}
When we discretize~\eqref{pLap} we obtain a system of nonlinear equations.
These equations inherit contraction properties from~\eqref{pLap}, namely stability in the maximum norm.   Our first goal is solve the equations, and our second goal is to solve them quickly, ideally in $\bO(N)$ iterations, where $N$ is the number of variables in the discrete system of equations.

The main issues to consider are stability of the iteration, and the convergence rate.
For implicit or semi-implicit schemes, we also need to  solve equations involving the operator.  This generally requires that the implicit operator be linear.

\subsection{Explicit methods}
There are no general methods available for solving non-divergence form nonlinear elliptic equations.   For monotone schemes, explicit methods can be used~\cite{ObermanSINUM}.
The explicit method in the case of~\eqref{pLap} is 
\[
\unn = \un + \dt (\pLap \un - g)
\]
where, to ensure stability, the artificial time step $\dt$ is restricted to be
the inverse of the Lipschitz constant of the scheme, regarded as a grid function.  The Lipschitz constant of the scheme is the inverse of the coefficient of $u(x)$ in the operator evaluated at $x$.  It is $\dt = \bO(\dx^2)$, usually $\dx^2/2$ although if a wide stencil is also used for $\Lap$ is can be made slightly larger.
Explicit methods are slow because of this restriction, which can be regarded as a nonlinear CFL condition.  The number of iterations required for convergence grows with the problem size.

Fully implicit methods require the solution of nonlinear equations.
In the case of~\eqref{pLap}, the fully implicit scheme
$
\frac{\unn - \un  }{\dt} = \pLap[\unn] -\grhs
$
is unconditionally stable, but requires solving the equation
\[
\unn - \rho \pLap \unn = \un - \dt\grhs.
\]
which has the same difficulties as solving the time-independent equation.

\subsection{Semi-implicit methods} 
We are motivated by the work~\cite{Smereka} which built a semi-implicit solver for the one-Laplacian, $\MC$,  by treating the linear part of the operator implicitly.

Write
\[
\IL = \frac 1 2  \left (  \Lap +  ( \IL -  \MC)  \right )
\]
which follows from~\eqref{LapILMC}.  This leads to the semi-implicit scheme

\bq\label{ILsi}
 -  \Lap \unn  =  -  ( \MC - \IL)  \un -   2\grhs.
\eq
The general case for $\pLap$ follows.

\begin{remark}
In practice, we only use the discretization of $\IL$ and the identity~\eqref{LapILMC}, so no discretization of $\MC$ is needed.
\end{remark}

\subsection{The semi-implicit scheme for $\pLap$}

Rewrite the equation~\eqref{pLapId} symmetrically in terms of $p,q$ as follows
\[
\pLap = \frac {p^{-1} + q^{-1}}{2}(\MC+ \IL) + \frac{p^{-1} - q^{-1}}2 ( \MC - \IL), 
\]
which gives
\[
\pLap = \frac 1 2  \left (  \Lap +  {\wb}( \MC - \IL)  \right ), \qquad \wb = 1 - \frac 2 p.
\]

The resulting scheme, which generalizes~\eqref{ILsi}, is given by
\bq\label{SI}\tag{Iteration}
 -\Lap \unn 
 =
\wb ( \IL - \MC) \un - 2\grhs, \qquad \wb = 1 - \frac 2 p.
\eq

The convergence of the iteration depends on whether the operator 
\[
\beta (-\Lap)^{-1}( \IL - \MC)
\]
is a contraction is some norm.  Clearly we can focus on the case
$\abs{\beta} = 1$ since the case  $\abs{\beta} < 1$ is easier.

\subsection{Proof of contraction for a linear model}
It would be desirable to prove that the iteration defined by~\eqref{SI} converges to the solution.  But to do so requires proving it is a contraction in some norm.  The operator involved is not monotone, so we are unable to prove it is a contraction in the uniform norm.  The nonlinearity makes proving convergence in other norms difficult.  However we present a heuristic for why the operator 
$
\Lap^{-1}( \IL - \MC)
$
is a contraction based on an analogy with a linear operator.  Instead of the operator above, we'll perform the analysis for 
\[
{ (-u_{xx} - u_{yy} )}^{-1}{ (-u_{xx} + u_{yy}) }
\]
which is a linear operator which is also a difference of degenerate elliptic operators. 
To be concrete, consider the case of the unit square in two dimensions. Use a simultaneous eigenfunction expansion,
\[
u = \sum_{i} a_{i} \phi_{i}
\]
with
\begin{align*}
M \phi_{i} &= \lambda_i^2 \phi_{i},  \quad \text{ where } M \equiv \partial_{xx}
\\
N \phi_{i} &= \nu_i^2 \phi_{i}, \quad \text{ where } N \equiv \partial_{yy}
\end{align*}
Then 
\[
\frac{ M - N }{ M + N} \phi_{i}  = \frac{ \lambda^2_i - \nu^2_i}{ \lambda^2_i + \nu^2_i} \phi_{i} \le \phi_{i}
\]
So
\[
\frac{ M - N }{ M + N} u = \frac{ M - N }{ M + N} \sum a_{i} \phi_{i}  = \sum a_{i} \frac{ \lambda^2_i - \nu^2_i}{ \lambda^2_i + \nu^2_i} \le \sum a_{i}\phi_{i} = u
\]
In fact, the contraction rate is given by 
\[
\max_{i,j} \frac{ \abs{\lambda^2_i - \nu^2_j}}{ \lambda^2_i + \nu^2_j} < 1, \quad \text{ on a finite grid }
\]
So in this simplified linear setting, the operator is a contraction.

\section{Numerical Results}

\subsection{Plots of solutions of $\IL$ with varying boundary data}
In~\autoref{figUsPlusLin} we plot solutions of the $\infty$-Laplacian, with variable boundary data, given by 
\[
\abs{x}-\abs{y} + c_i \frac{3x + 2y}{8\sqrt{14}}, \quad \text{ for } c_i = 0, 1,2, 3.
\]
Displayed are surface plots of the solution and  a contour plot of the norm of the gradient of the solutions. Note how the changing boundary data moves the kinks in the solution from a symmetric arrangement to a perturbed one.
The computational domain was given by $[-1,1]^2$ with a $200^2$ grid.

\begin{figure}[htbp]
\begin{center}
\includegraphics[width=.49\textwidth]{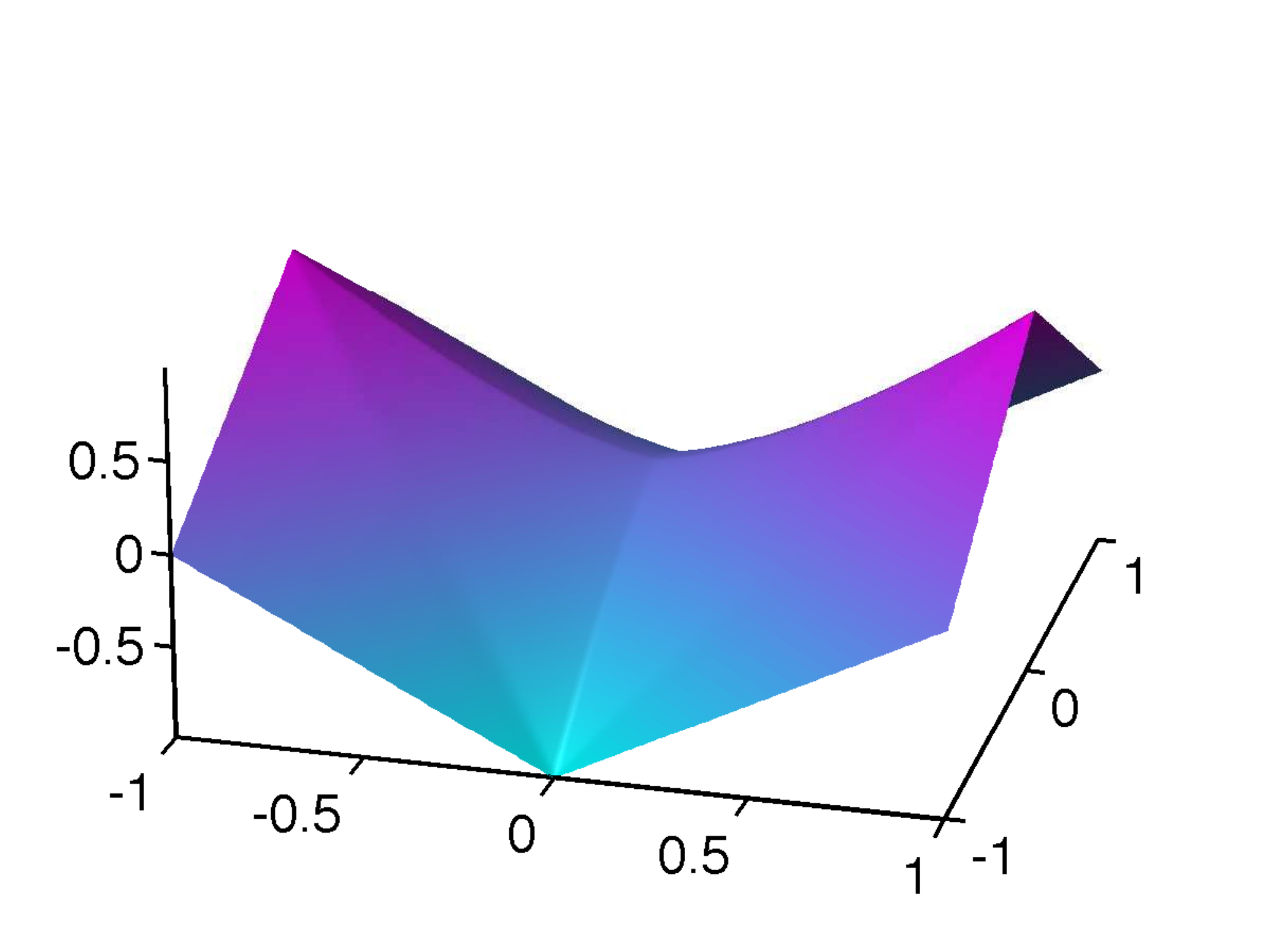}
\includegraphics[width=.49\textwidth]{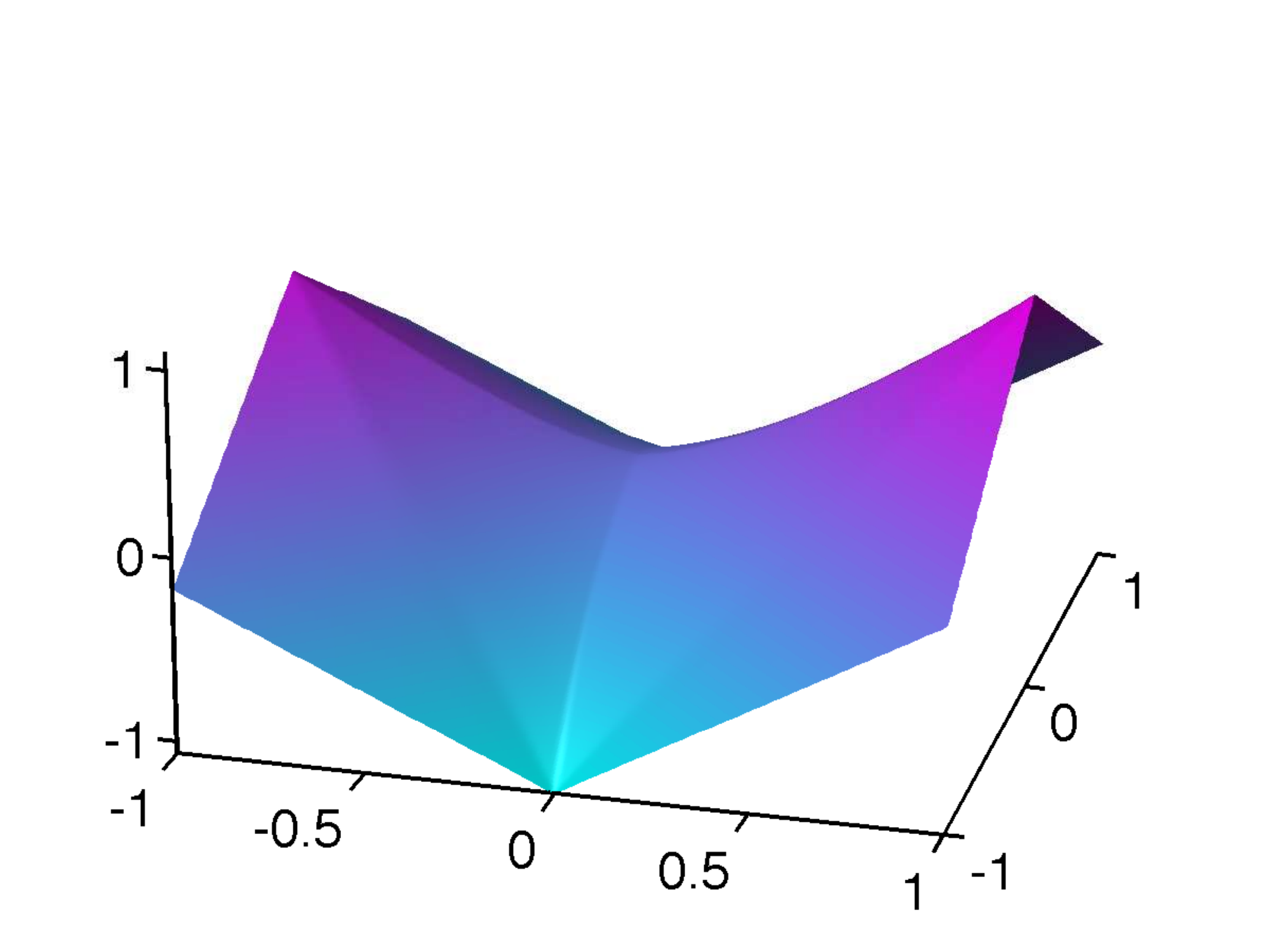}
\includegraphics[width=.49\textwidth]{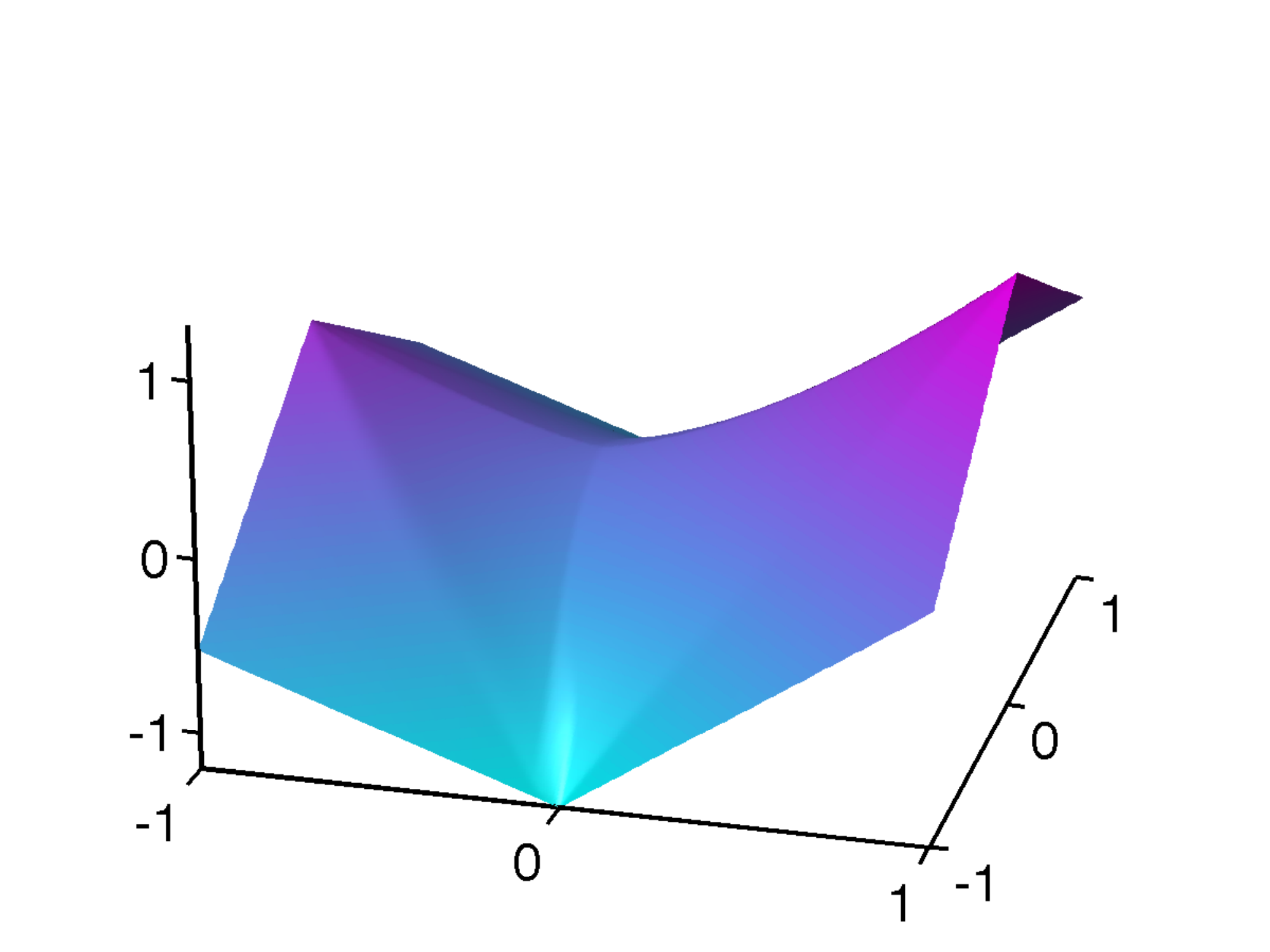}
\includegraphics[width=.49\textwidth]{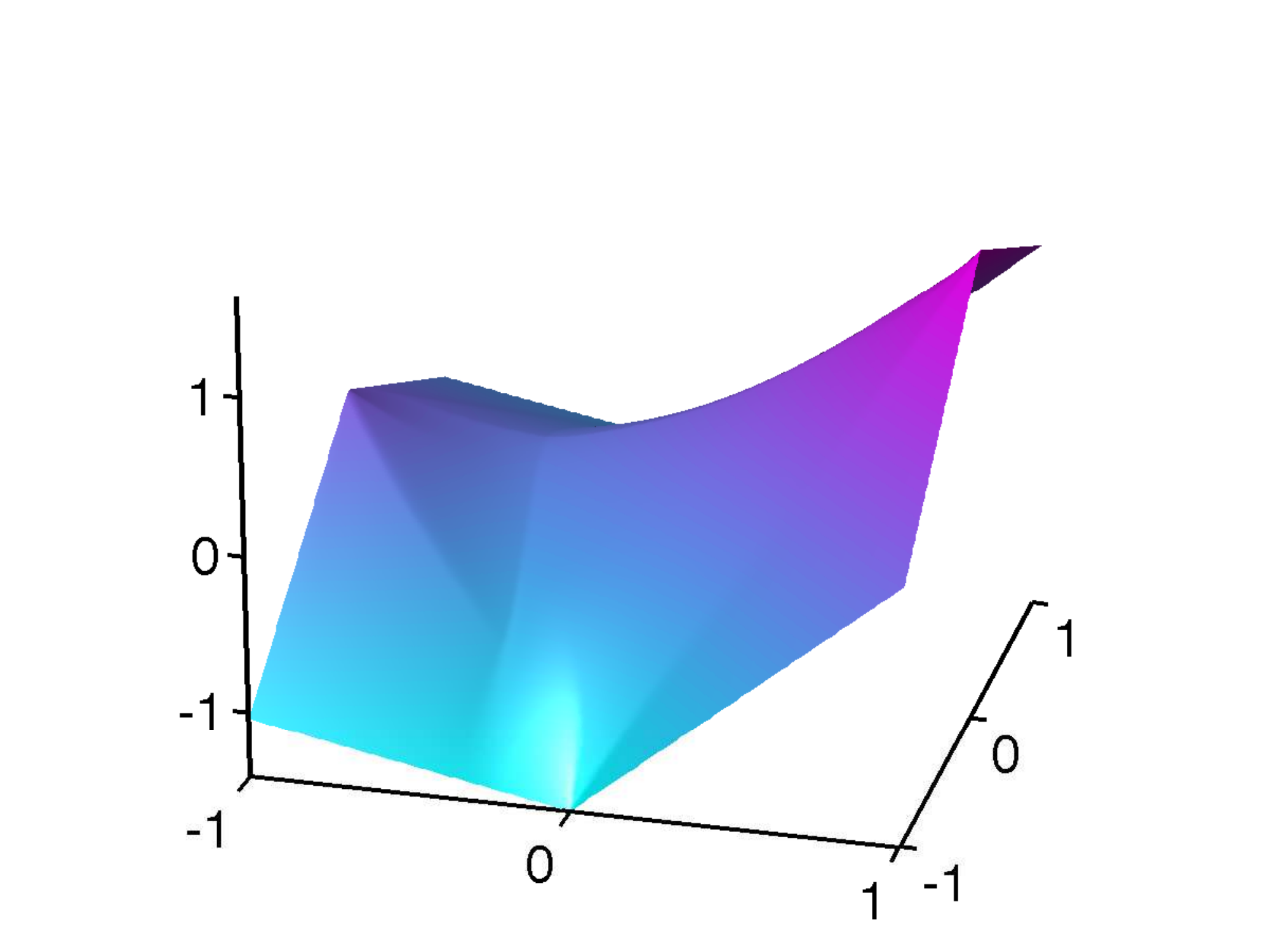}
\includegraphics[width=.85\textwidth]{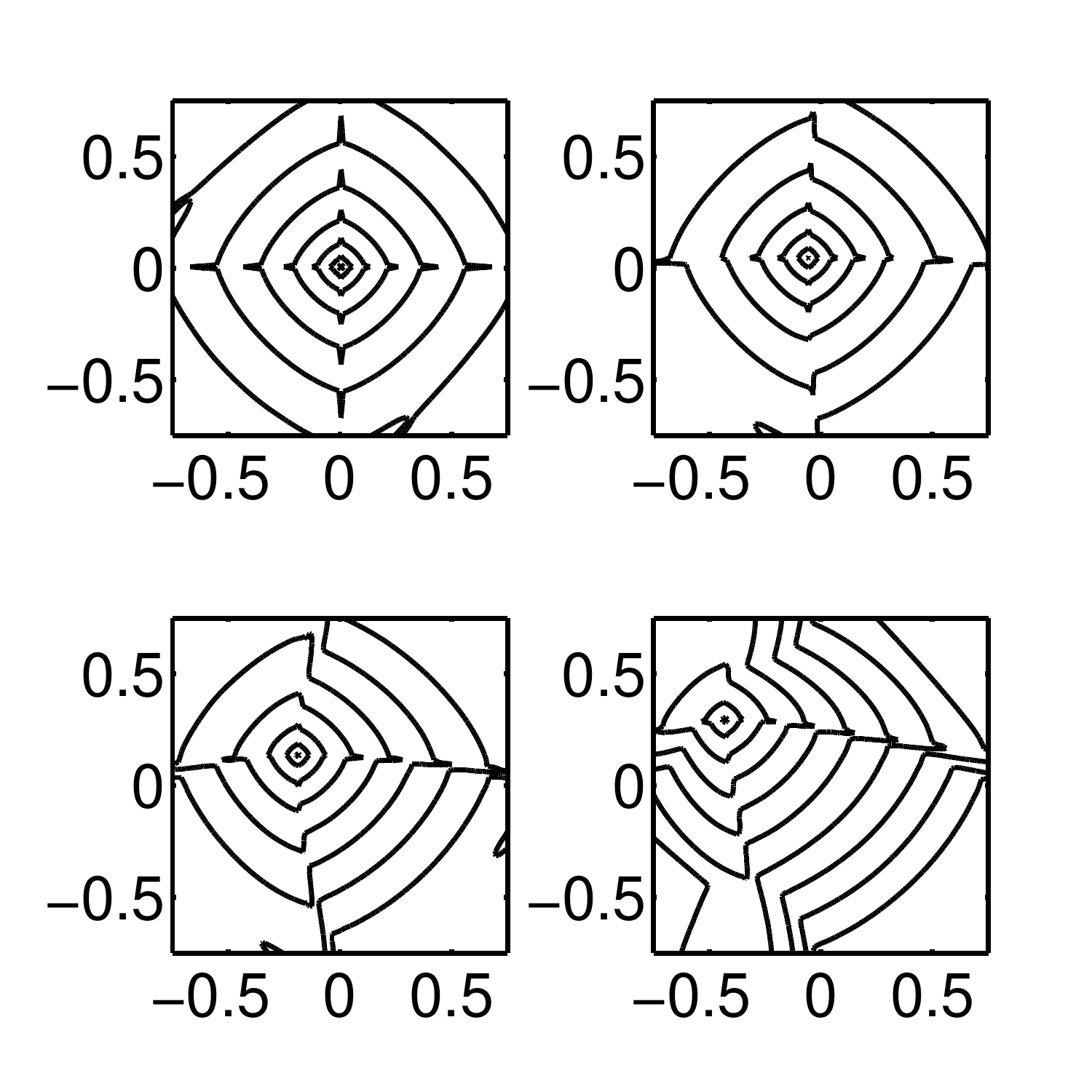}
\caption{
Solutions of the $\infty$-Laplacian.
Boundary data is a  $\abs{x}-\abs{y} + c_i (3x + 2y)/\sqrt{14}$, for  $c_i = 0, 1/8,2 /8, 3/8, 4/8$.
}
\label{figUsPlusLin}
\end{center}
\end{figure}

\subsection{Plots of solutions for varying $p$}
We solved the $p$-Laplacian on a $200^2$ grid, with values of $\wb = 0,1/3,2/3.1$, which corresponds to $p = 2,3,6,\infty$.    The boundary data was  $\abs{x}-\abs{y}$ on $[-1,1]^2$ . 

In \autoref{figUsSurf} we show the surface plots of the solution and  a contour plot of the norm of the gradient of the solution. 
The contours plotted are at levels sets equally spaced between 0 and 1.
 Note in the surface plot, the sharpening of the saddle to a kink in the gradient as $p$ increases.  In the contour plots, the contours get smaller as $p$ increases, showing that the solution is decreasing the size of the gradient.  Note also that the shape of the contours go from circular (for the Laplacian) to diamond shaped (for the Infinity Laplacian). Also note that the gradient in larger along the axes, which is where (presumably) the singularity in the solution is found.

\begin{figure}[htbp]
\begin{center}
\includegraphics[width=.49\textwidth]{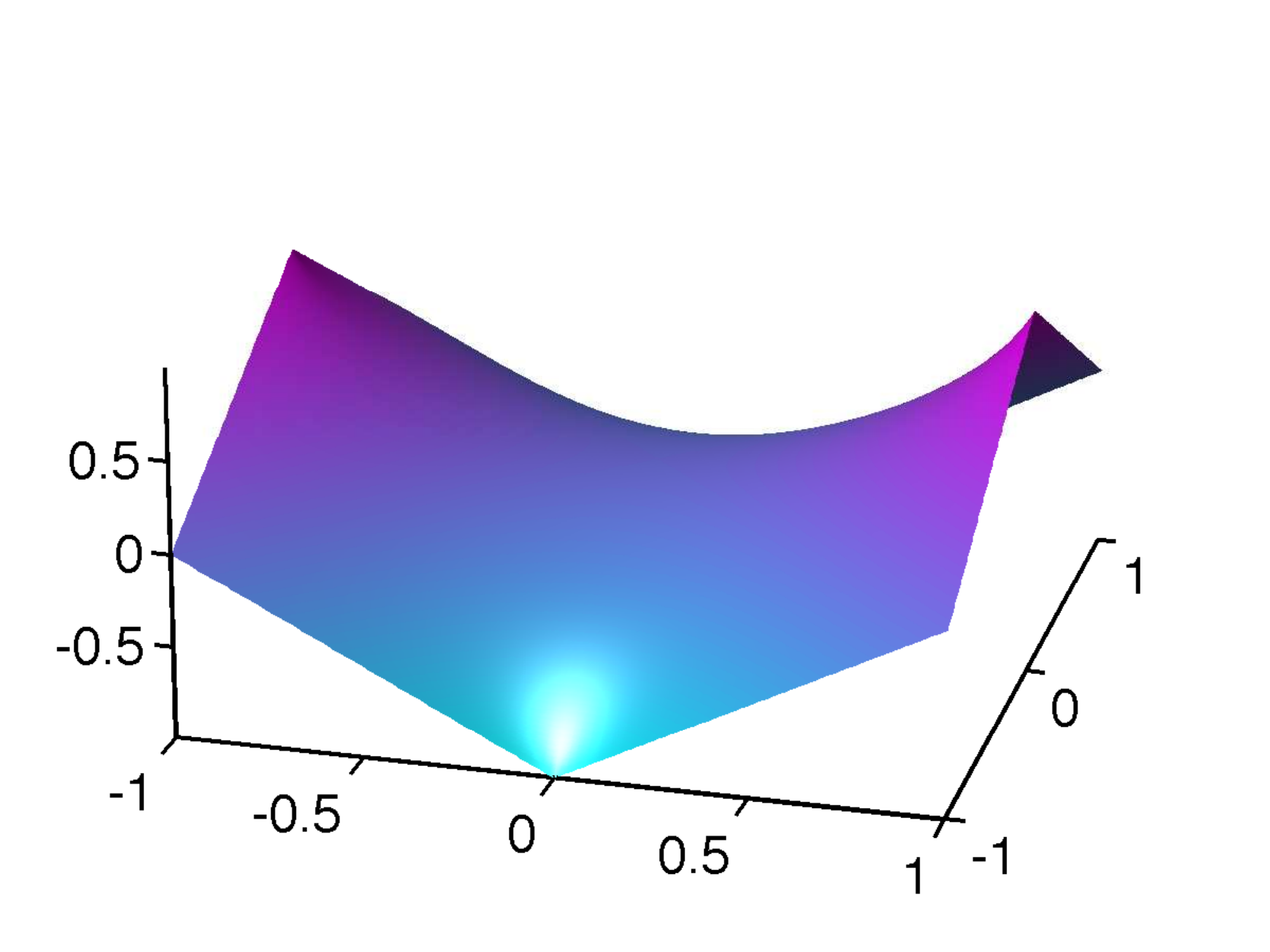}
\includegraphics[width=.49\textwidth]{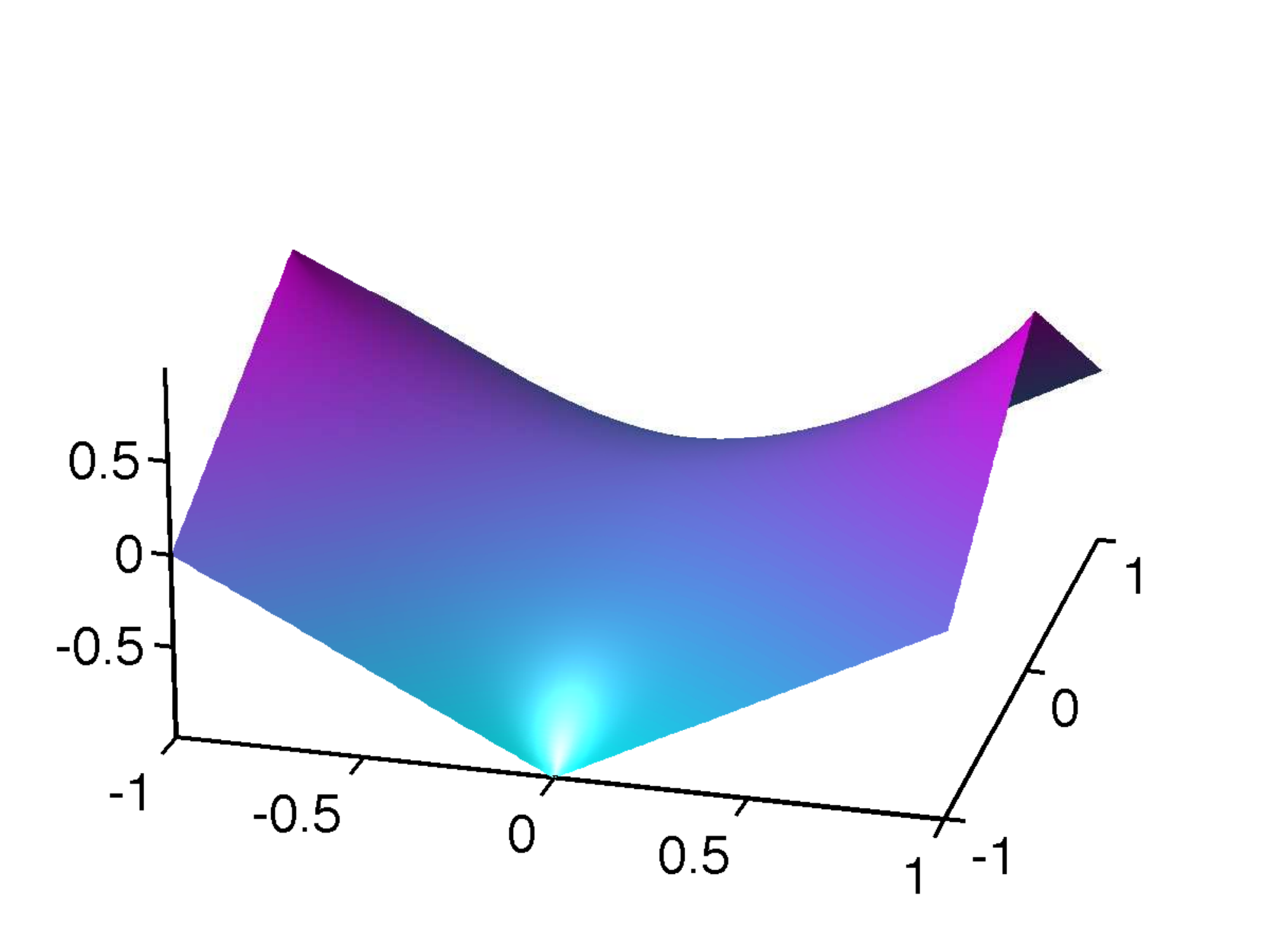}
\includegraphics[width=.49\textwidth]{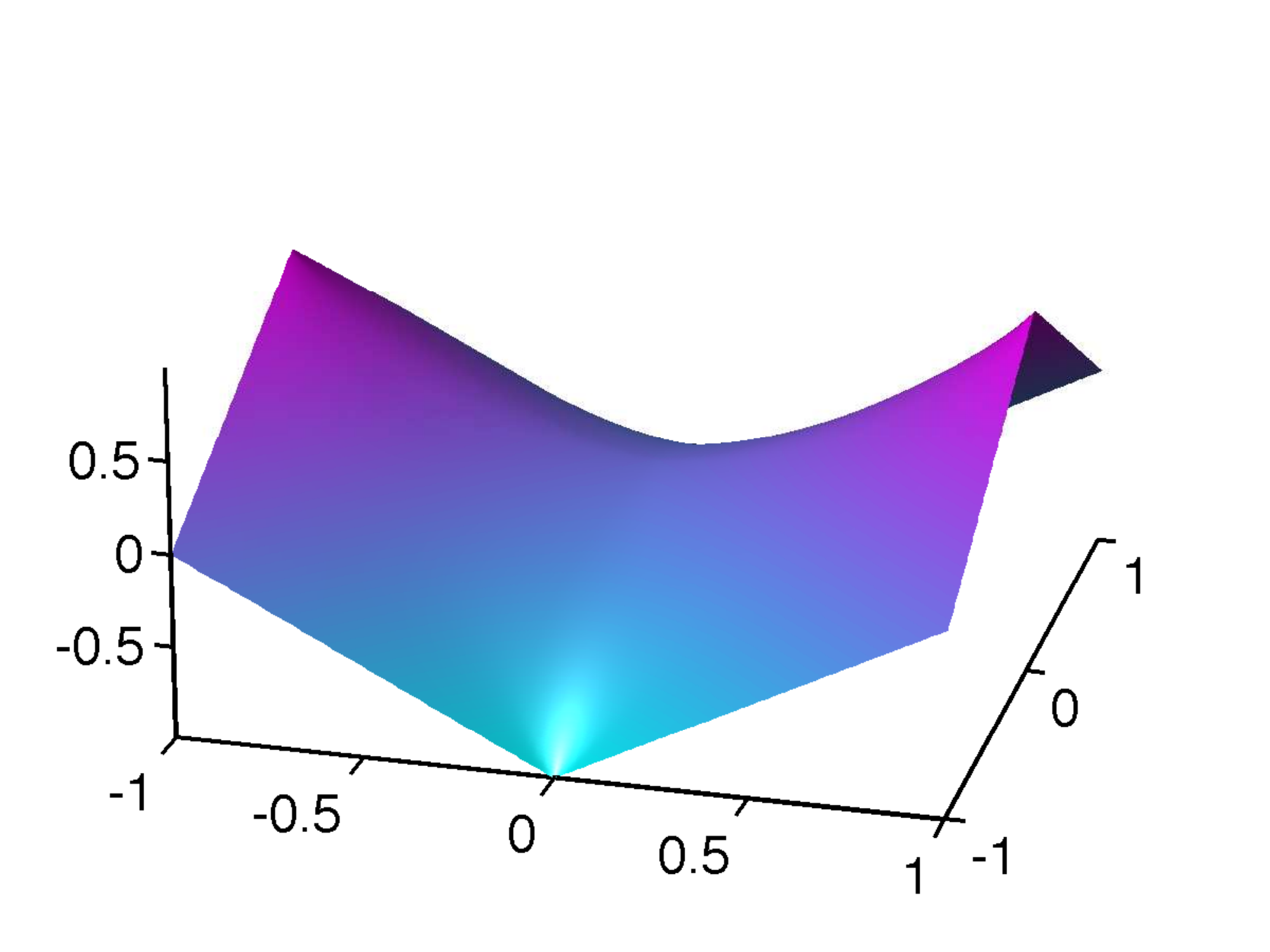}
\includegraphics[width=.49\textwidth]{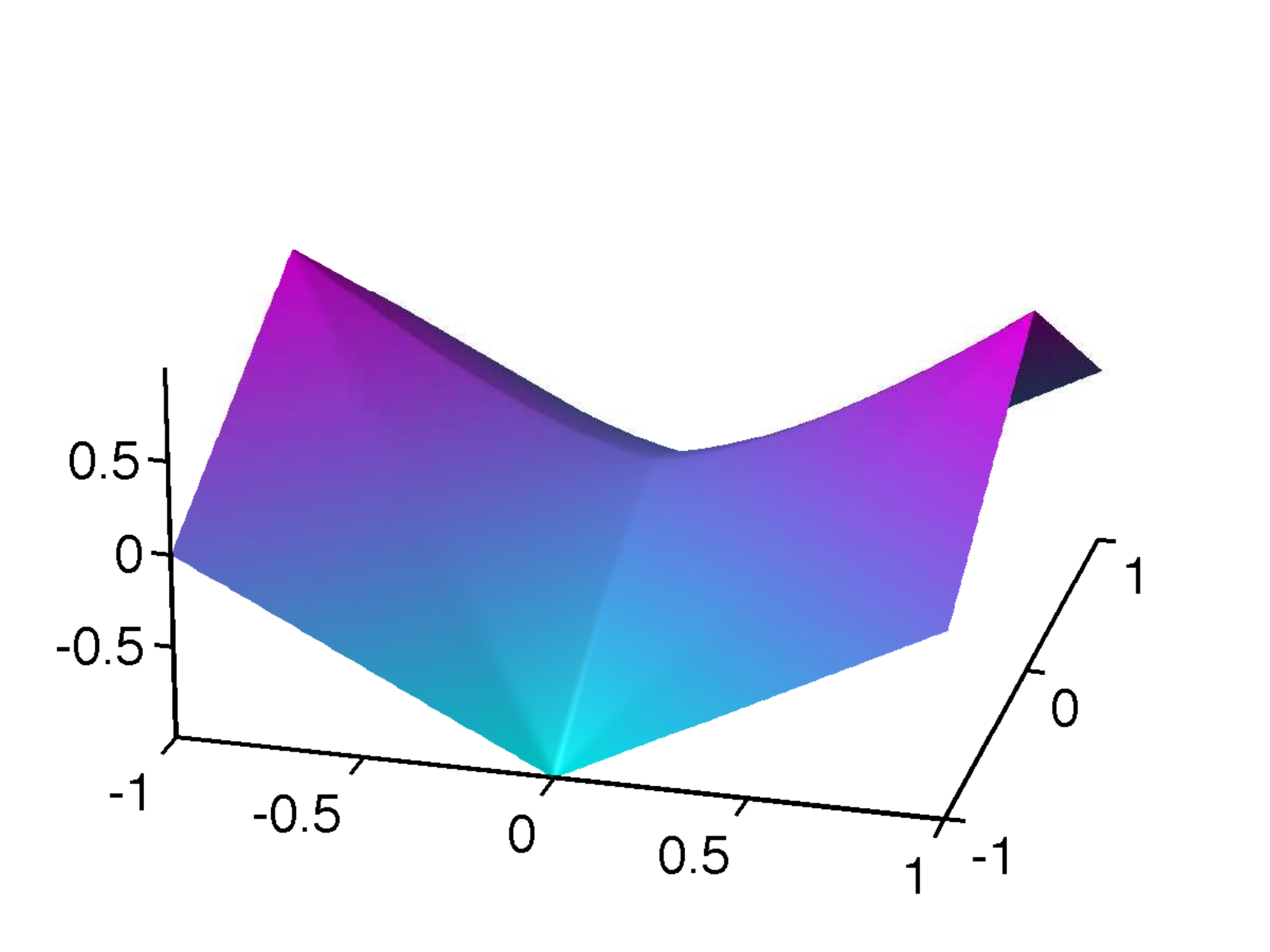}
\includegraphics[width=.85\textwidth]{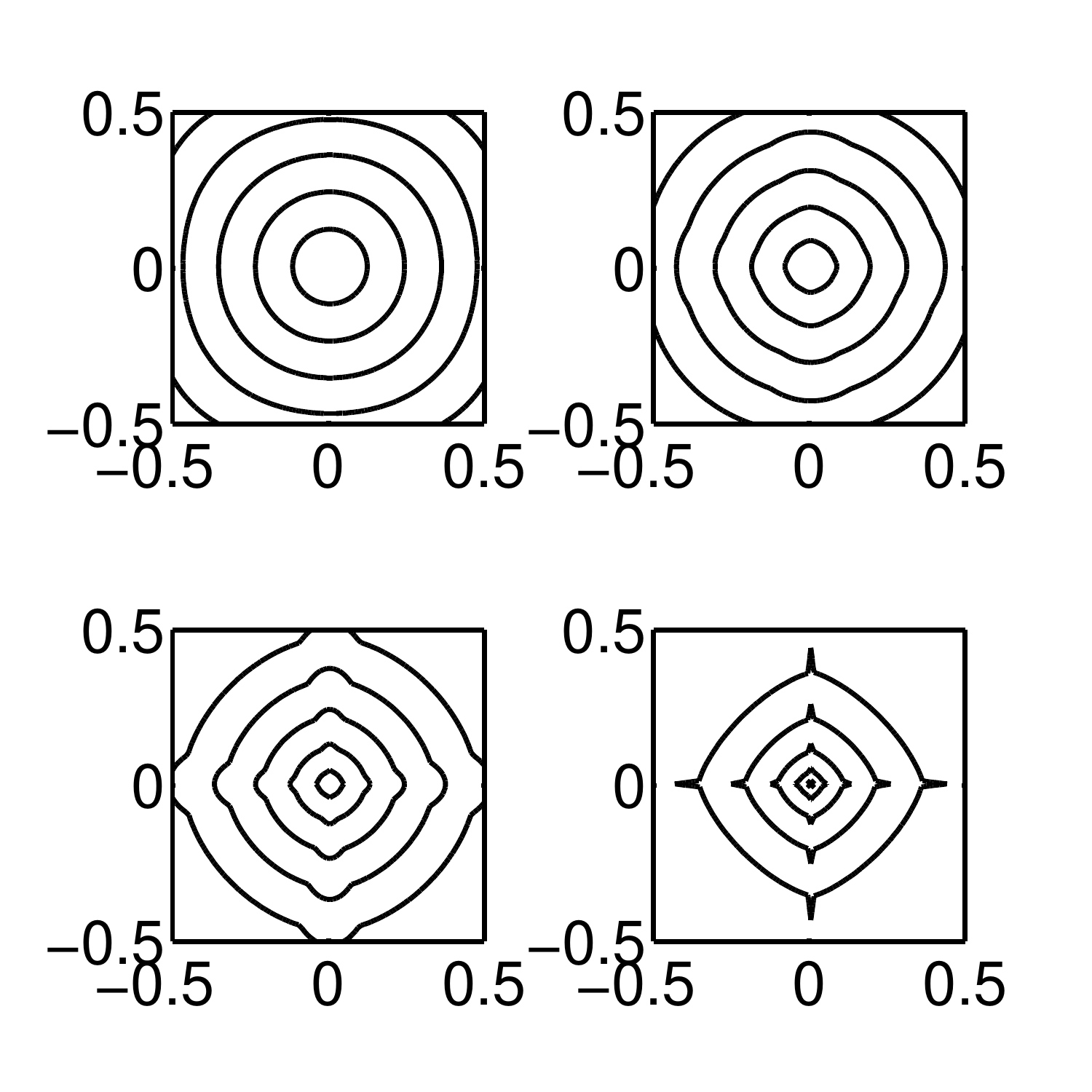}
\caption{
Solution of the $p$-Laplacian with $b = 0,1/3,2/3.1$, giving $p = 2,3,6,\infty$.
Boundary data is $\abs{x}-\abs{y}$. 
(a) Surface plot of the solutions.
(b) Contour plots of the norm of gradient of the solution.}
\label{figUsSurf}
\end{center}
\end{figure}

\section{Solver speed}
In this section we report on computational experiments which demonstrate the performance of the semi-implicit method.

The first section compares the semi-implicit method, which convergences at a rate which is independent of the problems, to the explicit method, which takes increasingly longer to converge as the problem size grow.

The next section gives further details on the solution of $\IL$.  Engineering precision is obtained in the first few iterations.  But numerical precision can still require many iterations.

The final section compares the speed of the solution of the $p$-Laplacian.  Setting $p< \infty$ improves the convergence rate of the solver.

\subsection{Speed of the explicit method.}  
The explicit method converges exponentially, but at a rate which depends on the problem size.
For small problem sizes, $n = 64$, we get an error of .1 after 200 iterations, and .001 after 1000 iterations.  But for large problems sizes, $n=512$, the error is $.7$ even after 1000 iterations.  This is prohibitively slow for larger problems sizes.

In~\autoref{figErrorsILExplicit} the error is shown as a function of the number of iterations, for different problem sizes.  Increasing the problem size means the error increases as well, for a fixed number of iterations.

We compared the semi-implicit method with the explicit method.
The error decreases faster for the semi-implicit method.  See \autoref{figErrorsILn128} for a typical example.

\begin{figure}[htbp]
\begin{center}
{
\includegraphics[width=.5\textwidth]{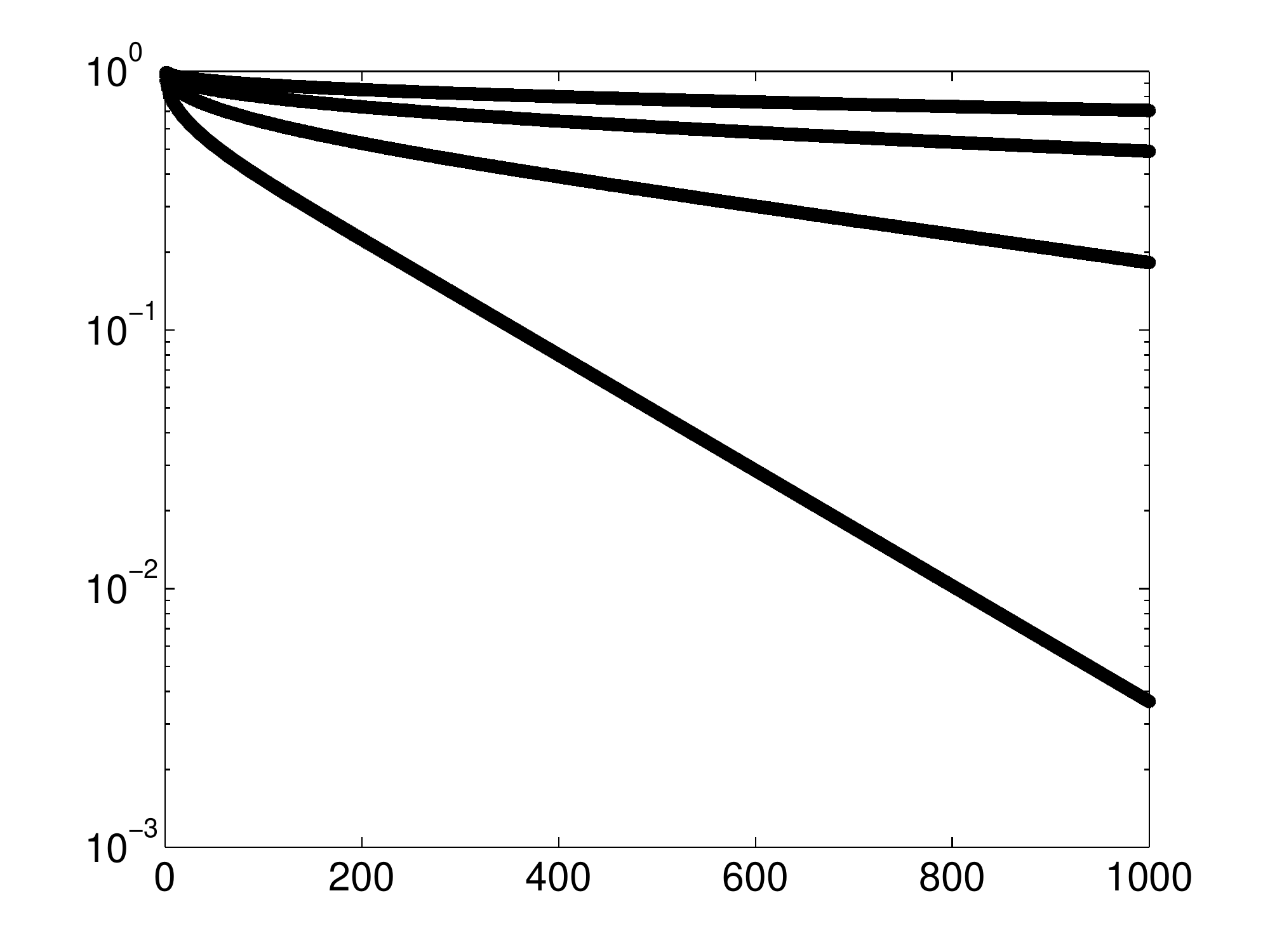}}
\caption{Maximum error as a function of the number of iterations in the solution of $\IL$ using the explicit method.   For $n = 64, 128, 256, 512$.   Errors are increasing with $n$.}
\label{figErrorsILExplicit}
\end{center}
\end{figure}

\begin{figure}[htbp]
\begin{center}
\includegraphics[width=.5\textwidth]{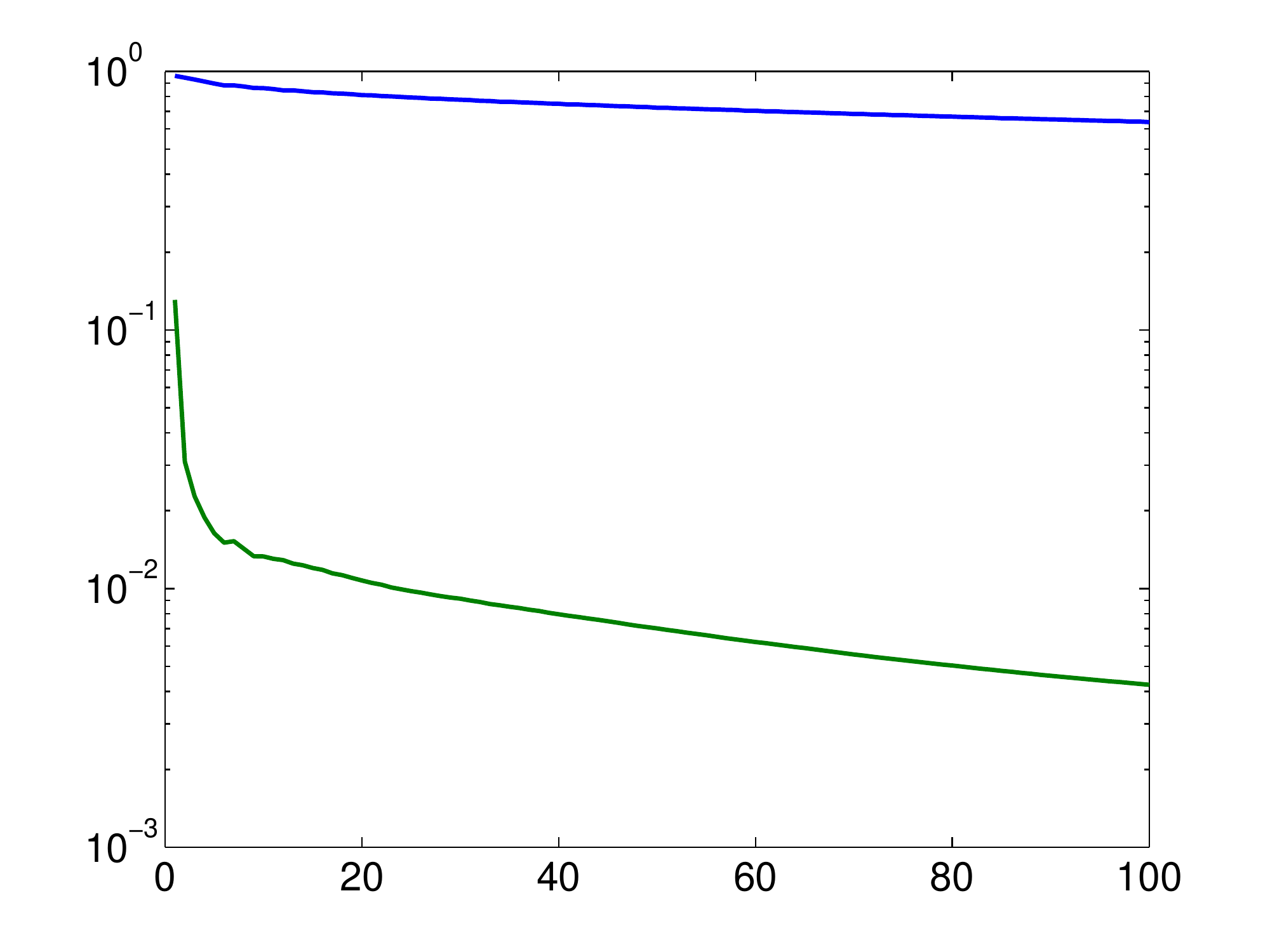}
\caption{Maximum error as a function of iterations for solutions of Infinity Laplace.  
Semi-implicit method (better accuracy) and Explicit method (less accurate), with $n=128$. Solution is $x^{4/3} - y^{4/3}$.}
\label{figErrorsILn128}
\end{center}
\end{figure}

\subsection{Solver speed for Infinity Laplace}
In the section we examine the performance of the semi-implicit method for solutions of Infinity Laplace.  

The advantage of the semi-implicit solver is that the convergence rate does not depend on the size of the problem.
Another advantage is that the semi-implicit solver reduces the error to engineering tolerances, $.1$, after one iteration, and achieves close to another order of magnitude after four iterations.  See \autoref{figErrorsIL}.   This is useful for applications, for example in imaging,  where high accuracy is not required.

\begin{figure}[htbp]
\begin{center}
{
\includegraphics[width=.5\textwidth]{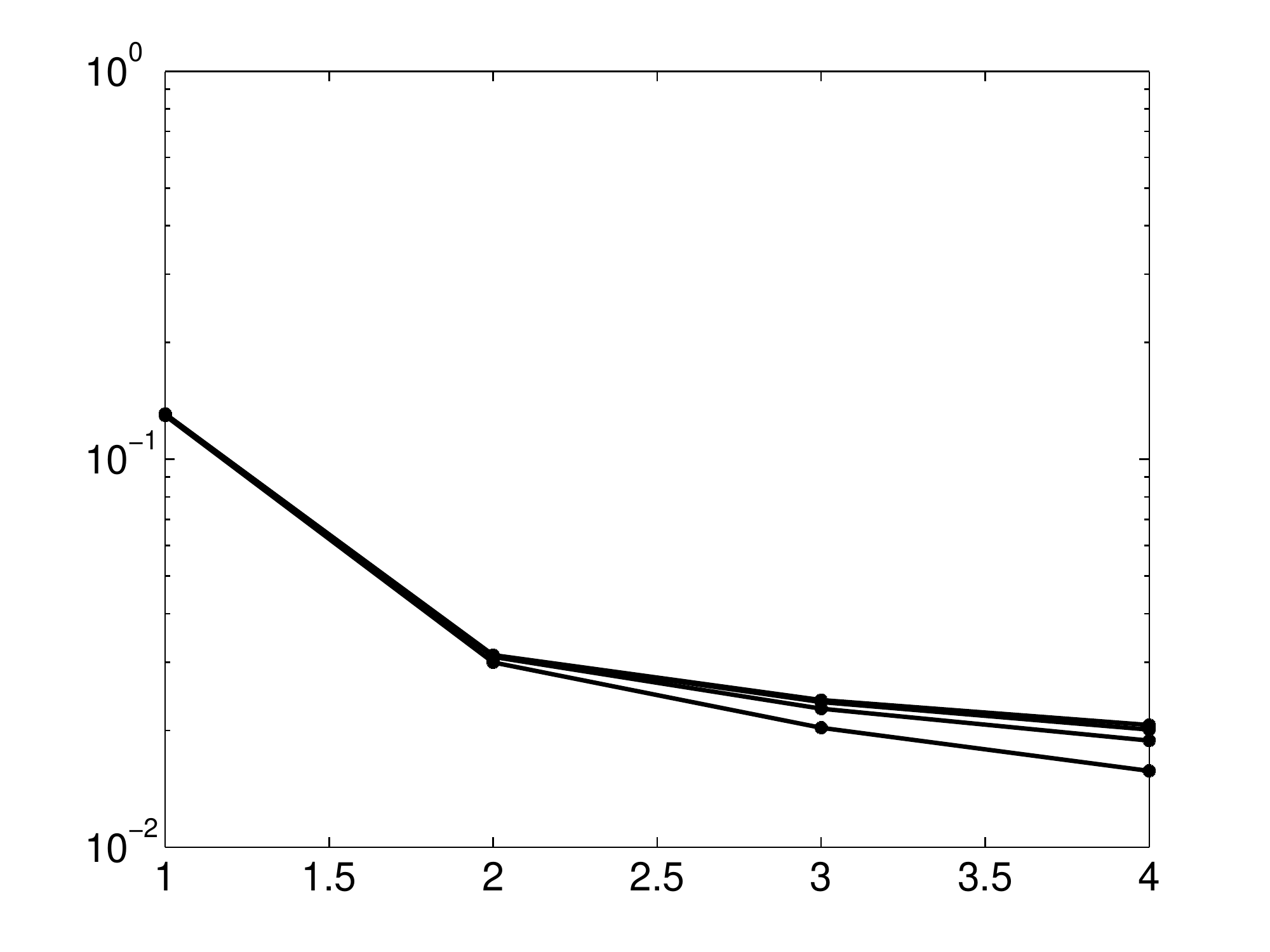}
}\caption{Maximum error as a function of the number of iterations in the solution of $\IL$ using the semi-implicit method. For $n = 64, 128, 256, 512$. Errors are nearly independent of $n$.}
\label{figErrorsIL}
\end{center}
\end{figure}

However, the specific rate of convergence varies for different solutions.
Comparing several solutions, it appears to be slowest for the solution $u(x,y) = x^{4/3} - y^{4/3}$, see \autoref{figErrorsIL3solns}

\begin{figure}[htbp]
\begin{center}
\includegraphics[width=.48\textwidth]{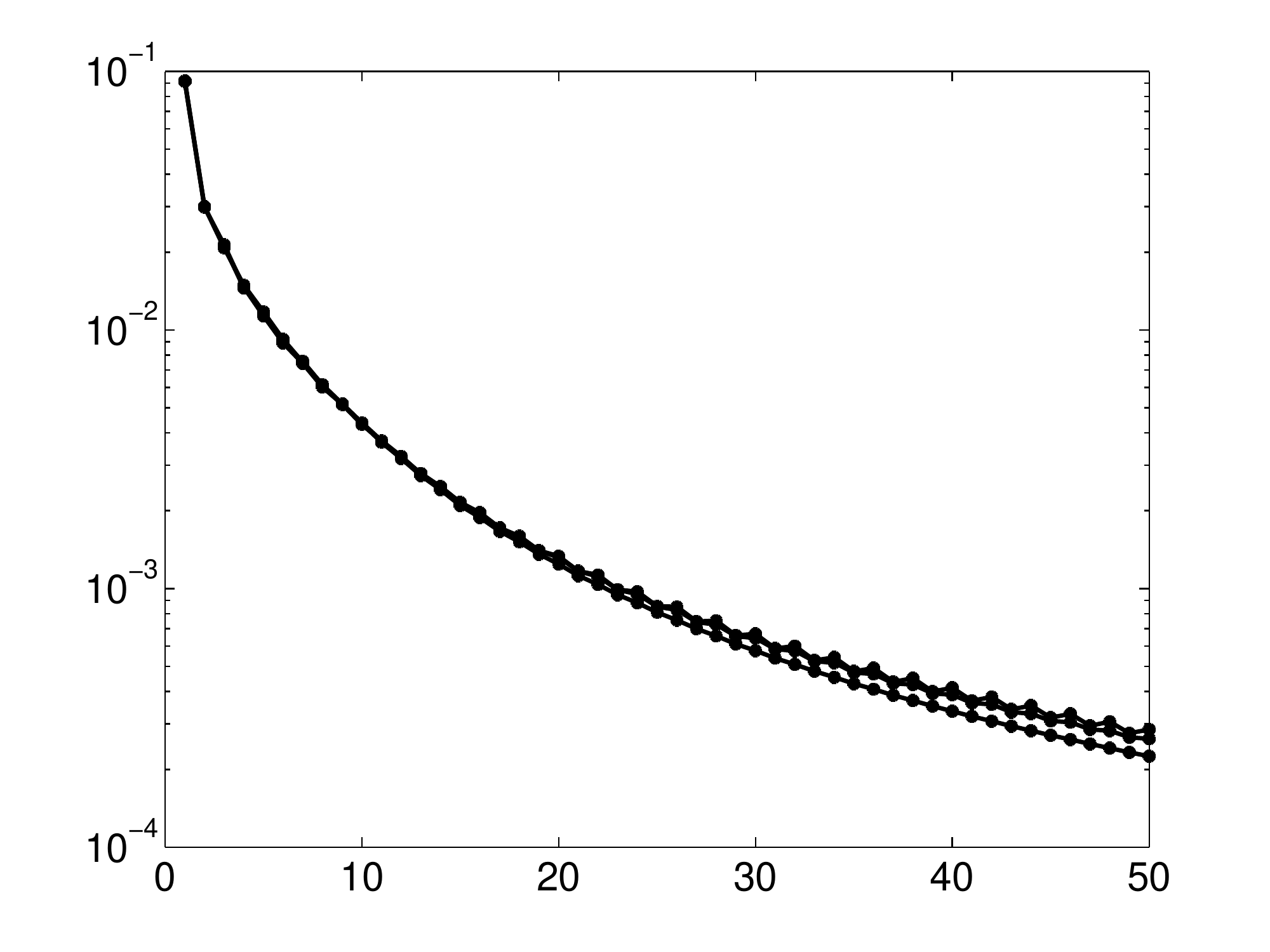}
\includegraphics[width=.48\textwidth]{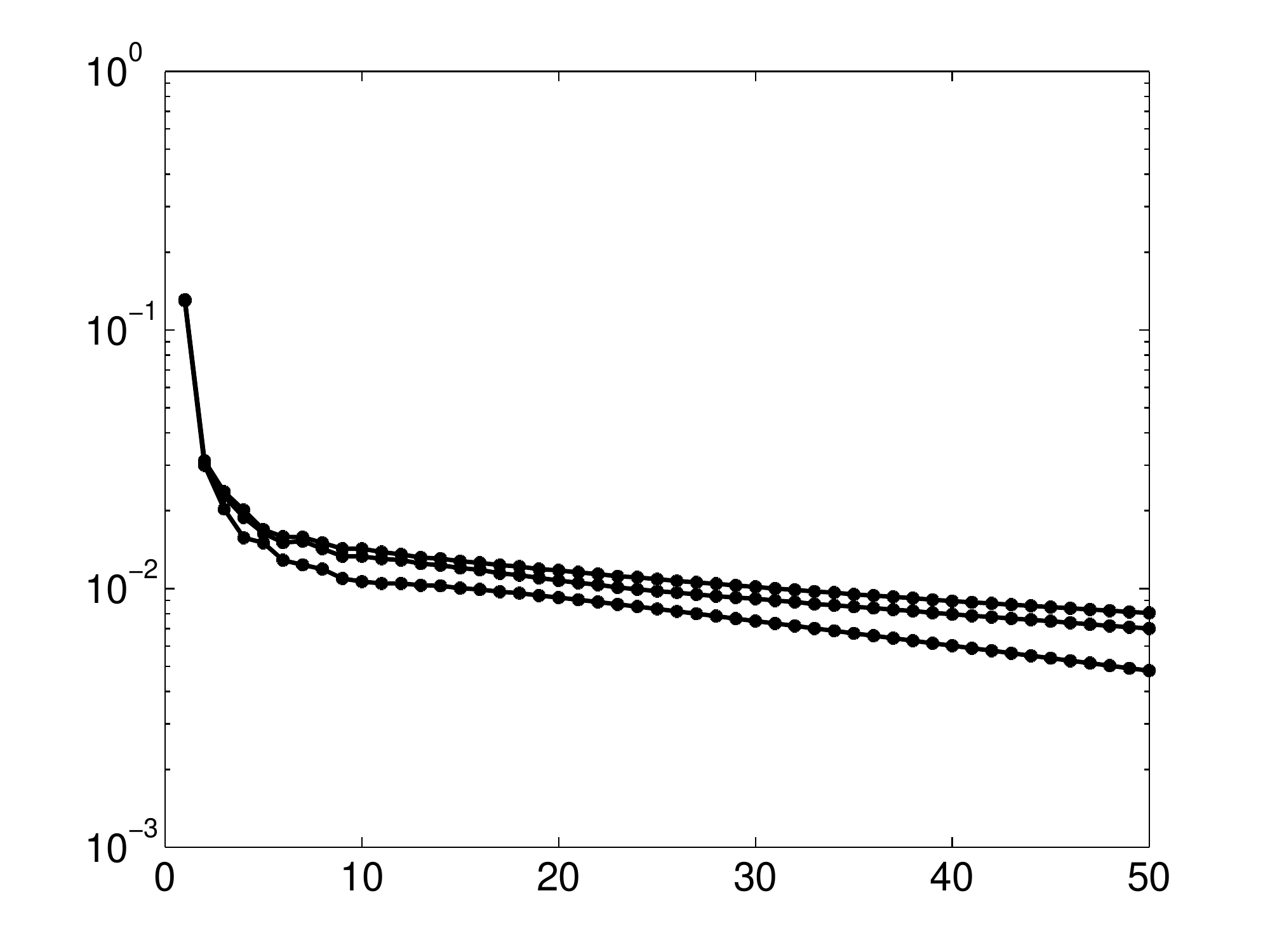}
\includegraphics[width=.48\textwidth]{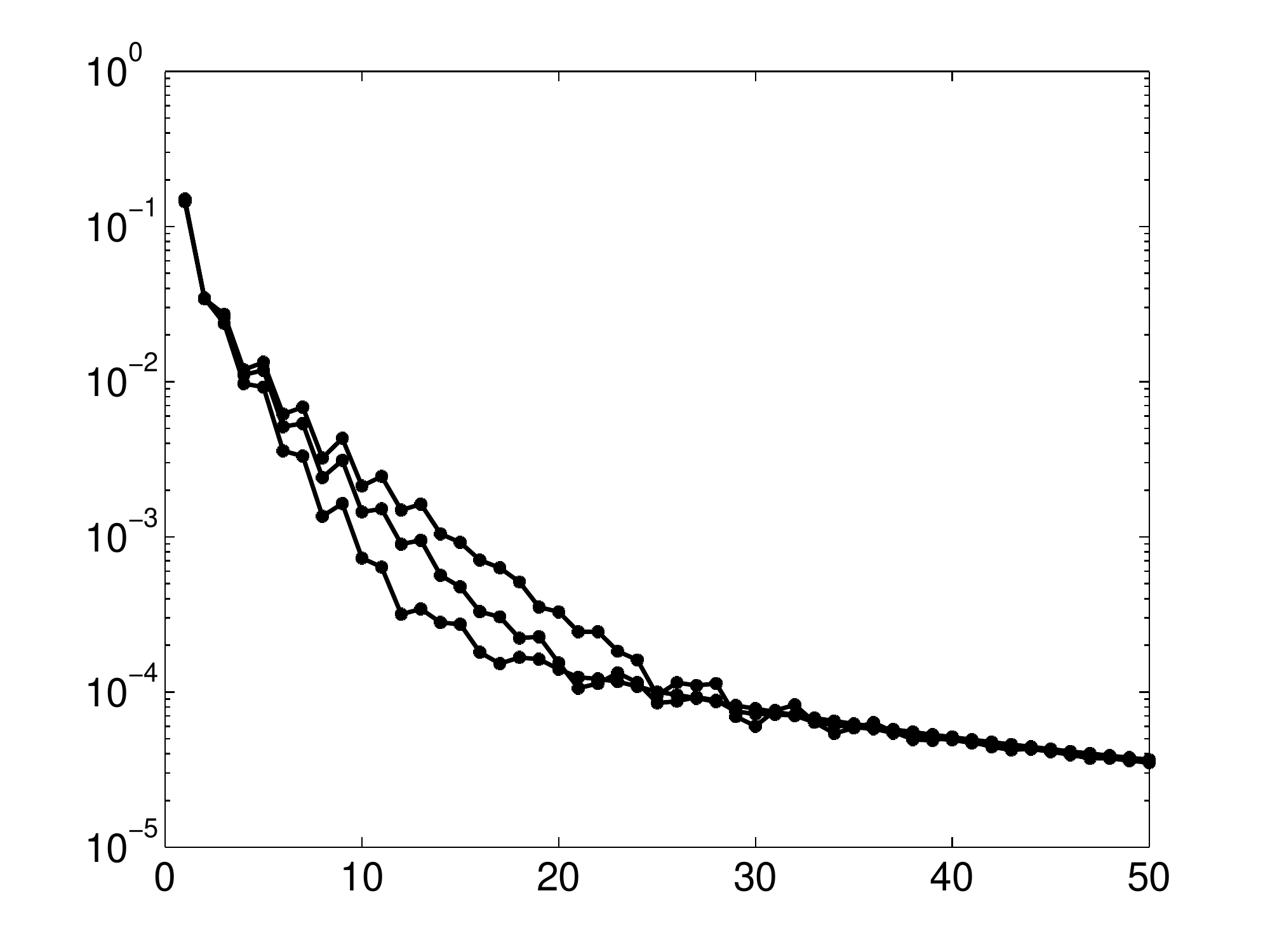}
\caption{Maximum error as a function of iterations for solutions of Infinity Laplace.  For $n = 64, 128, 256$.  Semi-implicit method. using exact solution of (a) $x^2 - y^2$ . (b) $x^{4/3} - y^{4/3}$ (c) $\abs{x} - \abs{y}$
}
\label{figErrorsIL3solns}
\end{center}
\end{figure}

\subsection{Solver speed for solutions of $p$-Laplacian}
We study the convergence rate of the iterative method for the $p$-Laplacian with $p<\infty$.  We also study the convergence rate in terms of problem size.  One notable fact is that the accuracy after just a few iterations is to engineering tolerances.

First we display the speed of convergence in terms of the number of iterations for different values of $\wa$, where $\wa = 1/p$ as given in~\eqref{plap2}.

The results which study the convergence rate for a particular example solution as a function of $\wa$ is presented in~\autoref{figConv} and~\autoref{figRate}.
 The convergence rate is independent of the size of the problem.   The convergence increases with $\wa$,
 which is to be expected, since the operator $\pLap$ is becoming better approximated by the Laplacian.   The convergence rate is consistent with an exponential error proportional to 
$10^{\mu(\wa) N}$, where 
the convergence rate $\mu(\alpha)$ is consistent with an affine expression
\[
\mu(\wa) =  -(c_1 + c_2 \alpha)  
\]
where $c_1$ is small and   
and $c_2 \approx 1$.


\begin{figure}[htbp]
\begin{center}
{ \includegraphics[width=.5\textwidth]{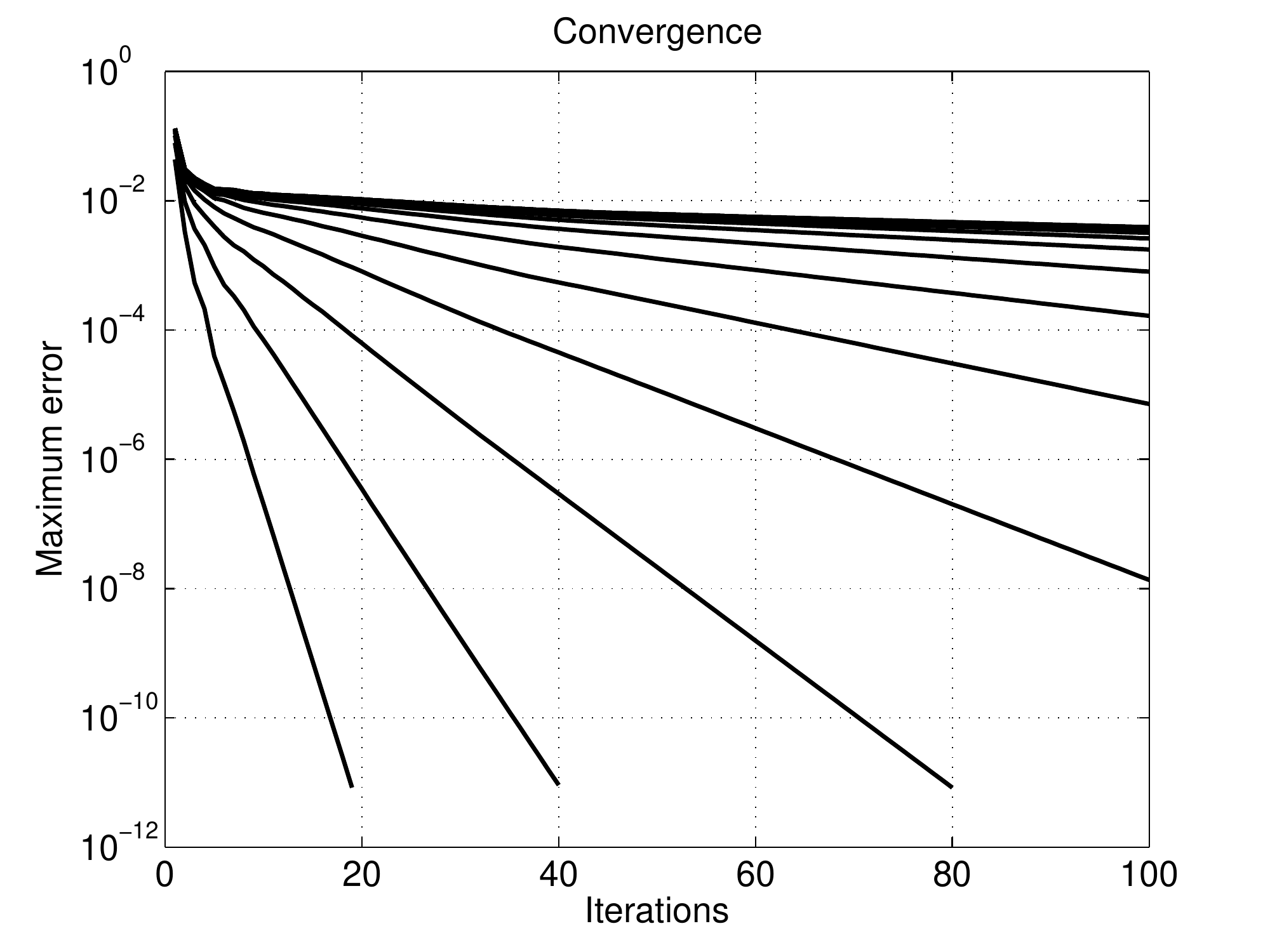}}

\caption{Convergence in the maximum norm as a function of the number of iterations, for $\wa = 1/2, 1/2^2, \dots, 1/2^{18}, 0$.
}
\label{figConv}
\end{center}
\end{figure}

\begin{figure}[htbp]
\begin{center}
{ \includegraphics[width=.5\textwidth]{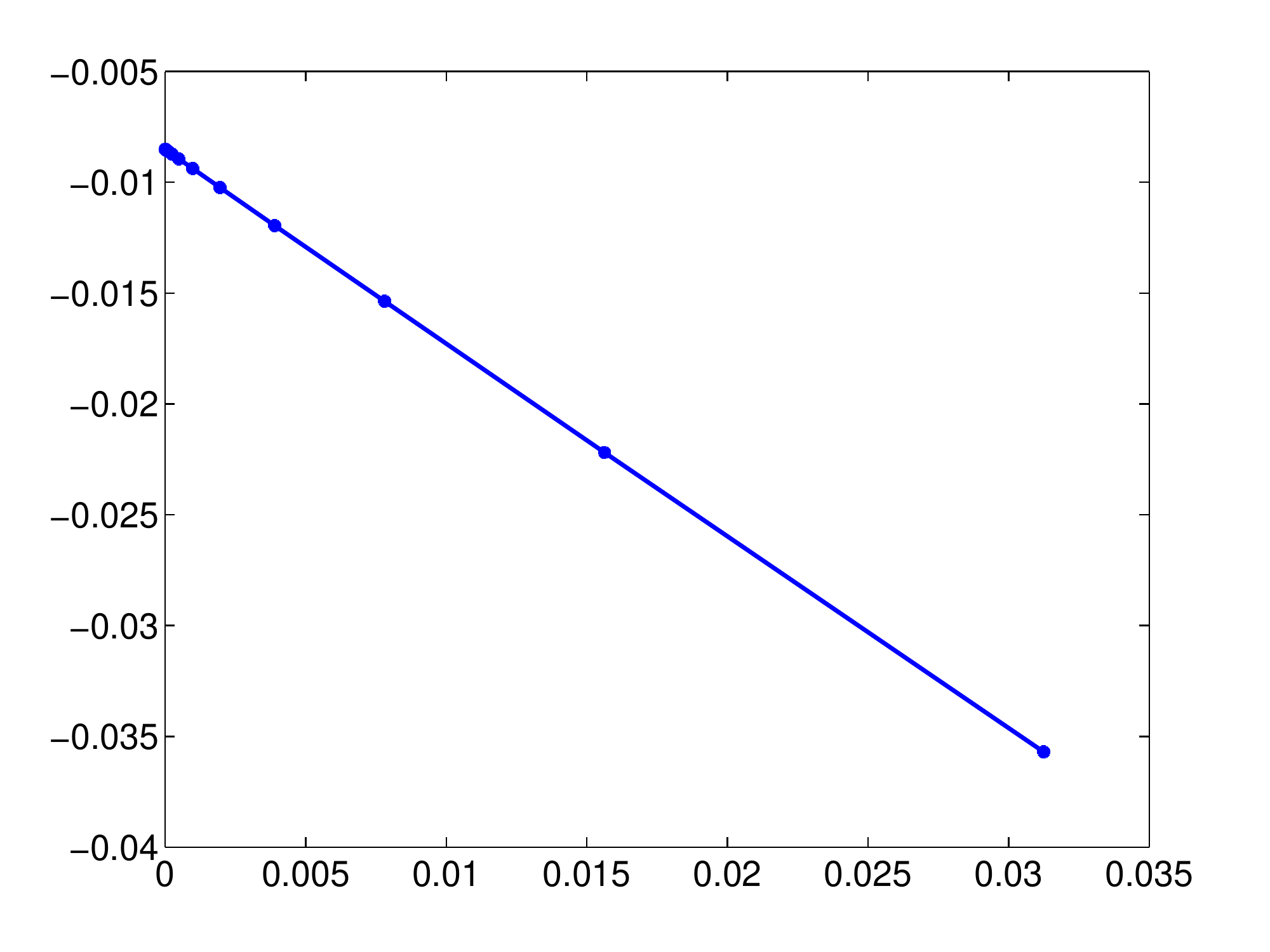}}

\caption{Convergence rate 
as a function of $\wa$ for
$\wa = 1/2, 1/2^2, \dots, 1/2^{18}, 0$.
The line of best fit is $-.87\wa - .0085$.
}
\label{figRate}
\end{center}
\end{figure}

%
%
%
%
%
%
%
%
\section{Conclusions}
We built convergent discretizations and fast solvers for the $\IL$ and $\pLap$ operators.

Since solutions of the PDE~\eqref{pLap} can be singular, it is important to use a convergent discretization.  The theory of viscosity solutions provides the appropriate notion of weak solutions for the PDE~\eqref{pLap}.   We simplified and generalized the Wide Stencil finite difference scheme for $\IL$ used in~\cite{ObermanIL} to build a monotone finite difference scheme for  $\pLap$.   The discretization took advantage of the fact the the $\pLap$ operator is a positive combination of $\Lap$ and $\IL$, given by~\eqref{plap2}.

A convergent discretization provides a finite dimensional equation whose solutions approximate the solutions of the PDE.  Proving convergence is established using PDE techniques, which does not address efficiency of solution methods.  Standard explicit methods are slow, in the sense that the number of iterations required to approximately solve the equations to a given level of precision grows with the problem size.  We introduced a semi-implicit solver, which converged at a rate independent of the problem size.  The idea for the solver is to approximate the operator $\pLap$ by the Laplacian, and solve a linear equation at each iteration.  While each iteration is more costly than the explicit method, the added cost is more than made up for by the fact that accurate solutions can be obtained in a few iterations.

Numerical results show that the method converges exponentially, at a rate which is independent of the problem size, even in the case $p = \infty$.
However, the convergence rate increases approximately linearly with $1/p $, which is to be expected, since as $p \to 2$ we recover the Laplacian and the solution is obtained in one step.  
For $\wa = 0$ or nearby,  engineering precision (maximum error of $10^{-2}$) is obtained in a few iterations, but the convergence to numerical precision is slower.  
For $\wa$ away from 0, the method is very fast, with numerical precision in about 20 iterations.
The convergence rate also depends on the particular solution.

\bibliographystyle{alpha}
\bibliography{pLap}

\end{document}